\documentclass{amsart}

\usepackage{geometry}
\usepackage{amsmath}
\usepackage{amssymb}
\usepackage{amsthm}
\usepackage{cleveref}
\crefname{subsection}{subsection}{subsections}
\crefname{lem}{lemma}{lemmas}
\usepackage{mathtools}
\usepackage{tikz}
\usepackage{comment}
\usetikzlibrary{patterns}
\usetikzlibrary{decorations.pathreplacing}

\usetikzlibrary{intersections}

\usepackage{float}

\usepackage{tkz-euclide}
\usetikzlibrary{calc,through}
\usetikzlibrary{angles,quotes}

\theoremstyle{definition}
\newtheorem{thm}{Theorem}[section]
\newtheorem{cor}[thm]{Corollary}
\newtheorem{prop}[thm]{Proposition}
\newtheorem{lem}[thm]{Lemma}

\newtheorem{quest}[thm]{Question}
\newtheorem{defn}[thm]{Definition}

\newtheorem{obs}[thm]{Observation}

\newtheorem{rmk}[thm]{Remark}

\newtheorem{claim}[thm]{Claim}
\newtheorem*{ack*}{Acknowledgements}

\newcommand{\pvh}{\textcolor{blue}}

\newcommand{\marius}{\textcolor{pink}}

\newcommand{\arc}{\mathfrak}

\newcommand{\co}{\operatorname{co}}

\newcommand{\hbad}{\mathcal{I}^{\text{bad}}_{100t^{-1}\ell}(\theta,\ell)}

\newcommand{\bad}{\mathcal{I}^{\text{bad}}_{3\ell}(\theta,\ell)}

\newcommand{\bads}{\mathcal{I}^{\text{bad}}_{s}(\theta,\ell)}

\newcommand{\jbads}{\mathcal{J}^{\text{bad}}_{s}(\theta,\ell)}

\newcommand{\jbad}{\mathcal{J}^{\text{bad}}_{3\ell}(\theta,\ell)}

\newcommand{\vbad}{\mathcal{I}^{\text{bad}}_{2\ell}(\theta,\ell)}

\newcommand{\jvbad}{\mathcal{J}^{\text{bad}}_{2\ell}(\theta,\ell)}

\newcommand{\rgood}{\mathcal{I}^{\text{good}}_{100t^{-1}\ell}(\theta,\ell)}

\newcommand{\good}{\mathcal{I}^{\text{good}}_{3\ell}(\theta,\ell)}

\newcommand{\goods}{\mathcal{I}^{\text{good}}_{s}(\theta,\ell)}

\newcommand{\jgoods}{\mathcal{J}^{\text{good}}_{s}(\theta,\ell)}

\newcommand{\jgood}{\mathcal{J}^{\text{good}}_{3\ell}(\theta,\ell)}

\newcommand{\pgood}{\mathcal{I}^{\text{good}}_{2\ell}(\theta,\ell)}

\newcommand{\jpgood}{\mathcal{J}^{\text{good}}_{2\ell}(\theta,\ell)}

\newcommand{\jj}{\mathcal{J}(\theta,\ell)}

\allowdisplaybreaks

\Crefname{subsection}{Section}{Sections}

\usepackage[colorinlistoftodos]{todonotes}

\begin{document}
\title{Sharp quantitative stability of the planar Brunn-Minkowski inequality}
\author{Peter van Hintum \and Hunter Spink \and Marius Tiba}
\email{pllv2@cam.ac.uk, hspink@math.harvard.edu, mt576@dpmms.cam.ac.uk}
\thanks{The authors would like to thank their respective institutions Clare College, University of Cambridge, Harvard University, and Trinity Hall, University of Cambridge. }

\begin{abstract}
    
    We prove a sharp stability result for the Brunn-Minkowski inequality for  $A,B\subset\mathbb{R}^2$. Assuming that the Brunn-Minkowski deficit $\delta=|A+B|^{\frac{1}{2}}/(|A|^\frac12+|B|^\frac12)-1$ is sufficiently small in terms of $t=|A|^{\frac{1}{2}}/(|A|^{\frac{1}{2}}+|B|^{\frac{1}{2}})$,  there exist homothetic convex sets $K_A \supset A$ and $K_B\supset B$ such that
    $\frac{|K_A\setminus A|}{|A|}+\frac{|K_B\setminus B|}{|B|} \le C t^{-\frac{1}{2}}\delta^{\frac{1}{2}}$. The key ingredient is to show for every $\epsilon>0$, if $\delta$ is sufficiently small then $|\co(A+B)\setminus (A+B)|\le (1+\epsilon)(|\co(A)\setminus A|+|\co(B)\setminus B|)$.
    
\end{abstract}

\maketitle

\section{Introduction}
Given measurable sets $A,B\subset \mathbb{R}^n$, the Brunn-Minkowski inequality says
$$|A+B|^{\frac{1}{n}} \ge |A|^{\frac{1}{n}}+|B|^{\frac{1}{n}},$$ with equality for homothetic convex sets $A=\co(A)$ and $B=\co(B)$ (less a measure $0$ set). Here $A+B=\{a+b\mid a\in A,\text{ and }b\in B\}$ is the \emph{Minkowski sum}, and  $|\cdot|$ refers to the outer Lebesgue measure. Stability results for the Brunn-Minkowski inequality quantify how close $A,B$ are to homothetic convex sets $K_A,K_B$ in terms of
\begin{itemize}
    \item $\delta=\delta(A,B):=\frac{|A+B|^{\frac{1}{n}}}{|A|^{\frac{1}{n}}+|B|^{\frac{1}{n}}}-1$, the \emph{Brunn-Minkowski deficit}, and
    \item $t=t(A,B):=\frac{|A|^{\frac{1}{n}}}{|A|^{\frac{1}{n}}+|B|^{\frac{1}{n}}}$, the \emph{normalized volume ratio}.
\end{itemize}
Throughout the paper, $\delta$ and $\tau$ will refer to the above quantities.

The sharp stability question for the Brunn-Minkowski inequality, \Cref{introquestion} below, is one of the central open problems in the study of geometric inequalities, and has been studied intensely in recent years by Barchiesi and Julin \cite{Barchiesi}, Carlen and Maggi \cite{Oneconvex}, Christ \cite{Christ}, Figalli and Jerison \cite{Fig,BM,Semisum}, Figalli, Maggi and Mooney \cite{Euclidean}, Figalli, Maggi and Pratelli \cite{Figalli09,Figalli10amass}, and the present authors \cite{PeterHunterMarius}. We provide a more detailed history of the problem in \Cref{introbackground}.
\begin{quest}
\label{introquestion}For $n \ge 1$, what are the optimal exponents $a_n$, $b_n$ (prioritized in this order) such that the following is true for every $\tau \in (0,\frac{1}{2}]$? There are constants $C_n$ and $d_n(\tau)>0$ such that whenever $A,B\subset \mathbb{R}^n$ are measurable sets with $t \in [\tau,1-\tau]$ and $\delta \le d_n(\tau)$, there exist homothetic convex sets $K_A\supset A$ and $K_B\supset B$ such that
$$\frac{|K_A\setminus A|}{|A|}+\frac{|K_B\setminus B|}{|B|}\le C_n\tau^{-b_n}\delta^{a_n}.$$
\end{quest}
We prioritize the exponents $a_n,b_n$ in this order as if the inequality holds for $(a_n,b_n)$, then the inequality also holds for $(a_n',b_n')$ whenever $a_n<a_n'$ by taking $d_n'(\tau)$ sufficiently small.

For planar regions, taking $A=[0,t]\times [0,t(1+\epsilon)]$ and $B=[0,(1-t)(1+\epsilon)]\times [0,1-t]$ shows that $a_2\le \frac{1}{2}$ and $b_2 \ge \frac{1}{2}$.
Our main result, \Cref{sharpBM}, solves the sharp stability question for planar regions $A,B\subset \mathbb{R}^2$, showing that the optimal exponents are $(a_2,b_2)=(\frac{1}{2},\frac{1}{2})$. 
\begin{thm}
\label{sharpBM}
There are constants $C,d(\tau)>0$ such that if $A,B\subset \mathbb{R}^2$ are measurable sets with $t\in [\tau,1-\tau]$ and $\delta \le d(\tau)$, then there are homothetic convex sets $K_A\supset A$ and $K_B \supset B$ such that
$$\frac{|K_A\setminus A|}{|A|}+\frac{|K_B\setminus B|}{|B|} \le C\tau^{-\frac{1}{2}}\delta^{\frac{1}{2}}.$$
\end{thm}


Our key result in proving \Cref{sharpBM} is a strong generalization to arbitrary sets $A,B$ of a conjecture \cite{Semisum} of Figalli and Jerison  for $A=B$  that $|\co(A)\setminus A|=O(\delta)$ for $\delta$ sufficiently small.  The original conjecture was recently proved by the present authors in \cite{PeterHunterMarius}. The generalization we now prove involves a completely different analysis to \cite{PeterHunterMarius}, and we are unaware of a similar approach used previously in the literature. 

\begin{thm}
\label{mainthm}
For all $\epsilon,\tau>0$ there is a constant $d_{\tau}(\epsilon)>0$ such that the following is true. Suppose that $A,B \subset \mathbb{R}^2$ are measurable sets with $t\in [\tau,1-\tau]$ and $\delta \le d_{\tau}(\epsilon)$. Then 
$$\left|\co(A+B)\setminus (A+B)\right| \le \left(1+\epsilon\right)\left(|\co(A)\setminus A|+|\co(B)\setminus B|\right).$$
\end{thm}
We note that \Cref{mainthm} with $\mathbb{R}^2$ replaced with $\mathbb{R}^n$ is false. Also,
taking $A=B=[0,1]^2 \cup\{(0,1+\lambda)\}$ shows that $(1+\epsilon)$ can't be replaced with anything smaller. We will ultimately prove \Cref{sharpBM} using \Cref{mainthm} with $\epsilon=\frac{\tau}{2}$.

\subsection{Background}
\label{introbackground}
In the literature, two measures  for quantifying how close $A,B$ are to homothetic convex sets have been introduced. The \emph{Fraenkel asymmetry index} is defined to be
$$\alpha(A,B)=\inf_{x \in \mathbb{R}^n}\frac{|A\mathrel{\Delta} (s\cdot \co(B)+x)|}{|A|}$$
where $s$ satisfies $|A|=|s\cdot \co(B)|$. The other measure introduced by Figalli and Jerison in \cite{BM} is $$\omega(A,B)=\min_{\substack{K_A \supset A, K_B\supset B\\K_A,K_B\text{ homothetic convex sets}}}\max\left\{\frac{|K_A\setminus  A|}{|A|},\frac{|K_B \setminus B|}{|B|}\right\}.$$ 
Providing an upper bound for $\omega$ is stronger than providing an upper bound for $\alpha$ as we always have $\alpha \le 2\omega$.
We note that in $\mathbb{R}^2$ when $A,B$ are both convex and $\delta$ is bounded, there is a reverse inequality (see \Cref{OmegaEquiv}).

In a landmark paper, Figalli and Jerison \cite[Theorem 1.3]{BM} showed the most general stability result for the Brunn-Minkowski inequality, with computable suboptimal exponents on $\tau$ and $\delta$, and with the exponent of $\delta$ depending on $\tau$ (which we rephrase for the convenience of the reader).
\begin{thm}(Figalli and Jerison \label{FigJerBm}\cite[Theorem 1.3]{BM})
There exist computable constants $a_n(\tau),b_n$ such that the following is true. There are computable constants $C_n$ and $d_n(\tau)>0$ such that whenever $A,B\subset \mathbb{R}^n$ with $t \in [\tau,1-\tau]$ and $\delta \le d_n(\tau)$, there exist homothetic convex sets $K_A\supset A$ and $K_B \supset B$ such that
$$\frac{|K_A\setminus A|}{|A|}+\frac{|K_B\setminus B|}{|B|} \le C_n \tau^{-b_n}\delta^{a_n(\tau)}.$$
\end{thm}
 This naturally gives rise to \Cref{introquestion}, asking for the optimal exponents of $\delta$ and $\tau$, prioritized in this order. This question, with $A,B$ restricted to various sub-classes of geometric objects, is the subject of a large body of literature. 
 Our main result \Cref{sharpBM} proves sharp stability in the case $n=2$ for arbitrary measurable $A,B$. 

Prior to \cite{BM}, Christ \cite{Christ} had proved a non-computable non-polynomial bound involving $\delta$ and $\tau$ via a compactness argument. When $A$ and $B$ are convex, the optimal inequality $\alpha\le C_n\tau^{-\frac{1}{2}}\delta^{\frac{1}{2}}$ was obtained by Figalli, Maggi, and Pratelli in \cite{Figalli09,Figalli10amass}. When $B$ is a ball and $A$ is arbitrary, the optimal inequality $\alpha\le C_n\tau^{-\frac{1}{2}}\delta^{\frac{1}{2}}$ was obtained by Figalli, Maggi, and Mooney in \cite{Euclidean}. We note that this particular case is intimately connected with stability for the isoperimetric inequality. When just $B$ is convex the (non-optimal) inequality 
$\alpha\le C_n\tau^{-(n+\frac{3}{4})}\delta^{\frac{1}{4}}$ was obtained by Carlen and Maggi in \cite{Oneconvex}. Finally, Barchiesi and Julin \cite{Barchiesi} showed that when just $B$ is convex, we have the optimal inequality $\alpha\le C_n\tau^{-\frac{1}{2}}\delta^{\frac{1}{2}}$, subsuming these previous results.

Before their general result for distinct sets $A,B$ in \cite {BM}, Figalli and Jerison \cite{Fig} had considered the case $A=B$ and gave a polynomial upper bound $\omega \le  C_n\delta^{a_n}$. Later, in \cite{Semisum}, they conjectured the sharp bound $\omega \le C_n\delta$ when $A=B$, and proved it in dimensions 2 and 3 using an intricate analysis which unfortunately does not extend to higher dimensions. Afterwards, Figalli and Jerison suggested a stronger conjecture that $\omega \le C_n\tau^{-1}\delta$ for $A,B$ homothetic regions, which was proved by the present authors in \cite{PeterHunterMarius}.

\subsection{Outline of Paper}\label{Outline}


In \Cref{Setup}, we give a reformulation of \Cref{mainthm}, make some simplifications and general observations, and give definitions which will be used throughout the remainder of the paper. Simplifications include assuming $A,B$ are finite unions of polygonal regions so the vertices of $\partial\co(A),\partial\co(B)$ are contained in $A,B$ respectively, and that they are translated in a specific way so that $\co(A)$ and $\co(B)$ contain the origin $o$.

In \Cref{GuaranteedRegionsSection}, by an averaging argument we show that $(1-4\tau^{-1}\sqrt{\gamma})\co(A+B)\subset A+B$, where $\gamma=|\co(A)\setminus A|+|\co(B)\setminus B|$, i.e. for every $x \in \partial \co(A+B)$, we have $(1-4\tau^{-1}\sqrt{\gamma})ox \subset A+B$.

In \Cref{decgoodbadsection}, we introduce a partition of $\partial \co(A+B)$ into \textbf{good} arcs and \textbf{bad} arcs. We think of good arcs as being the parts of the boundary of $\co(A+B)$ which are straight (or close to straight). We show that a very small part of the boundary $\partial \co(A+B)$ is covered by bad arcs.

In \Cref{xirootgammasection}, we show for $x$ in a good arc of $\partial \co(A+B)$, we can in fact guarantee that $(1-\xi\sqrt{\gamma})ox$ lies in $A+B$ for any small $\xi$ (provided small $d_\tau$). Thus $\co(A+B)\setminus A+B$ lies in a thickened boundary $\Lambda$ of $\partial \co(A+B)$, which is thinner near the good arcs.

In \Cref{paralelegramcoveringsection,preimagesABsection}, we set up the following method for proving $|\co(A+B)\setminus (A+B)| \le (1+\epsilon)(|\co(A)\setminus A|+|\co(B)\setminus B|)$.

The edges of $\partial \co(A+B)$ are precisely the edges of $\partial \co(A)$ and $\partial \co(B)$ attached one after the other ordered by slope. Moreover, every edge of $\partial \co(A+B)$ is the Minkowski sum of an edge of $\partial \co(A)$ with a vertex of $\partial \co(B)$ or vice versa. We subdivide $\partial\co(A+B)$ into tiny stright arcs $\mathcal{J}$, and partition these arcs into collections $\mathcal{A}$ and $\mathcal{B}$ accordingly. We note that the arcs of $\mathcal{A}$ can be reassembled to $\partial\co(A)$ and the arcs of $\mathcal{B}$ can be reassembled to $\partial\co(B)$, in the same orders as they appear in $\partial \co(A+B)$.

We erect on each arc $\arc{q}\in \mathcal{J}$ a parallelogram $R_{\arc{q}}$ pointing roughly towards the origin such that these parallelograms cover the thickened boundary $\Lambda$. We ensure that we use a constant number of directions ($1000$ suffices), such that the $R_{\arc{q}}$s with the same directions occur in contiguous arcs of $\partial \co(A+B)$. The heights of the parallelograms will be roughly on the order of $\sqrt{\gamma}$  if $\arc{q}$ lies in a bad arc, and $\xi\sqrt{\gamma}$ if $\arc{q}$ lies in a good arc. Each parallelogram $R_{\arc{q}}$ with $\arc{q}\in\mathcal{A}$ is the Minkowski sum of a parallelogram $R_{\arc{q},A}$ erected on the corresponding segment of $\partial \co(A)$ with a vertex $p_{\arc{q},B}\in \partial \co(B)\cap B$. Similarly for $\arc{q}\in \mathcal{B}$.

This construction allows us to cover the thickened boundary $\Lambda$ of $\partial \co(A+B)$ with translates of small regions erected on $\partial \co(A)$ and $\partial co(B)$ as follows:
\begin{align*}
    \Lambda \subset \bigcup_{\arc{q}\in \mathcal{A}} (R_{\arc{q},A}+p_{\arc{q},B})\cup \bigcup_{\arc{q}\in \mathcal{B}} (p_{\arc{q},A}+R_{\arc{q},B}).
    \end{align*}
    Therefore, we can cover $\co(A+B)\setminus (A+B)$ as follows:
\begin{align*}
    \co(A+B)\setminus (A+B) \subset \bigcup_{\arc{q}\in \mathcal{A}} ((R_{\arc{q},A}\setminus A)+p_{\arc{q},B})\cup \bigcup_{\arc{q}\in \mathcal{B}} (p_{\arc{q},A}+(R_{\arc{q},B}\setminus B)).
\end{align*}
If we have subsets $\mathcal{A}'\subset \mathcal{A}$ and $\mathcal{B}'\subset \mathcal{B}$ such that $\{R_{\arc{q},A}\}_{\arc{q}\in \mathcal{A}'}$ are disjoint and contained in $\co(A)$ and analogously $\{R_{\arc{q},A}\}_{\arc{q}\in \mathcal{B}'}$ are disjoint and contained in $\co(B)$, then we obtain an inequality
$$|\co(A+B)\setminus A+B| \le |\co(A)\setminus A|+|\co(B)\setminus B|+\sum_{\arc{q}\in \mathcal{A}\setminus \mathcal{A}'}|R_{\arc{q},A}|+\sum_{\arc{q}\in \mathcal{B}\setminus \mathcal{B}'}|R_{\arc{q},B}|.$$
Hence to prove \Cref{mainthm}, it suffices to show that we can find such $\mathcal{A}'$ and $\mathcal{B}'$ with $$\sum_{\arc{q}\in \mathcal{A}\setminus \mathcal{A}'}|R_{\arc{q},A}|+\sum_{\arc{q}\in \mathcal{B}\setminus \mathcal{B}'}|R_{\arc{q},B}| \le \epsilon (|\co(A)\setminus A|+|\co(B)\setminus B|).$$
In \Cref{farweightedaveragessection} we show that bad arcs of $\partial \co(A+B)$ are close in angular distance to the corresponding arcs in $\partial \co(A)$ and $\partial \co(B)$. This result is crucial for \Cref{juttingsection,boundingoverlappingsection} where we bound the areas of the parallelograms we have to remove to create $\mathcal{A}'$ and $\mathcal{B}'$.

In \Cref{juttingsection}, we use  \Cref{farweightedaveragessection} to show that parallelograms $R_{\arc{q},A} \not \subset \co(A)$ and $R_{\arc{q},B} \not \subset B$ have $\arc{q}$ on a good arc. This is then used to show that the area of parallelograms not contained in $\co(A)$ or $\co(B)$ is bounded roughly by $\xi^2 \gamma$.


In \Cref{boundingoverlappingsection} we use \Cref{farweightedaveragessection} to show that parallelograms $R_{\arc{q},A}$ and $R_{\arc{r},A}$ that intersect non-trivially have at least one of $\arc{q}$ and $\arc{r}$ on a good arc. This allows us to remove only good parallelograms to ensure disjointness. We conclude that the area of parallelograms we need to remove is bounded by roughly $\xi \gamma$. 

In
\Cref{puttingtogethersection} we complete the proof of \Cref{mainthm} by synthesizing our bounds to deduce the final inequality. In \Cref{implicationsection} we show how \Cref{mainthm} implies \Cref{sharpBM}. Finally, we add an appendix with the proof that the measures $\alpha$ and $\omega$ are commensurate.

\subsection{Acknowledgements}
The authors would like to thank their supervisor Professor B\'ela Bollob\'as for his continuous support.

\section{Setup}
\label{Setup}
In this section, we collect together the preliminaries we need to start proving \Cref{mainthm}. In \Cref{equalareaformulationsubsection} we introduce an equal area reformulation of \Cref{mainthm}. In \Cref{prelimaffinesubsection} we apply a preliminary affine transformation to $\mathbb{R}^2$ and collect facts about the resulting lengths and areas. In \Cref{defsubsection} we collect the main definitions which will be used throughout the body of the paper. Finally, in \Cref{genobssubsection} we collect general observations which we will use frequently throughout.
\subsection{Equal area reformulation}
\label{equalareaformulationsubsection}
We will primarily work with the equivalent equal area reformulation \Cref{mainthm'} of \Cref{mainthm}.
\begin{defn}
For $A,B\subset \mathbb{R}^2$ measurable sets and $t \in [0,1]$, define 
$$D_t=tA+(1-t)B.$$
\end{defn}
\begin{thm}
\label{mainthm'}
For $\tau \in (0,\frac{1}{2}]$, there are constants $d_{\tau}=d_{\tau}(\epsilon)>0$ such that the following is true. Let $A,B\subset \mathbb{R}^2$ be measurable sets with $|A|=|B|=V$, let $t$ a parameter satisfying $t \in [\tau,\frac{1}{2}]$, and suppose that $|D_t|\le(1+ d_{\tau}(\epsilon))^2V$. Then
$$|\co(D_t)\setminus D_t| \le (1+\epsilon)(t^2 |\co(A)\setminus A|+(1-t)^2 |\co(B)\setminus B|).$$
\end{thm} In \Cref{mainthm'}, $t$ is a free parameter, which we note is the normalized volume ratio of $tA$ and $(1-t)B$. Given the sets $A,B$ in \Cref{mainthm}, $A/t$ and $B/(1-t)$ have equal volumes, and \Cref{mainthm} is equivalent to \Cref{mainthm'} applied with these equal volume sets.

In the equal area reformulation, we let $K$ be the smallest convex set such that $K$ contains a translate of $A$ and $B$. We assume from now on that $A,B\subset K$. By approximation\footnote{It is easy to show that for any fixed $d_\tau(\epsilon)$ we must have $A,B$ bounded. Now, approximate $A,B$ from the inside by a nested sequence of compact subsets $A_1\subset A_2\subset \ldots$ and $B_1\subset B_2\subset \ldots$. Then for each $A_i,B_i$ approximate the pair from the outside by finite unions of polygons.}, we may assume that $A,B,K$ are unions of polygons.

\subsection{Preliminary affine transformation}
\label{prelimaffinesubsection} Let $T\subset K$ be the maximal area triangle, and let $o$ be the barycenter (which we will always take to be the origin). This maximal area triangle $T$ has the property that $T\subset K \subset -2T:=T'$, and by applying an affine transformation, we may assume that $T$ is a unit equilateral triangle whose vertices are contained in $K$. 
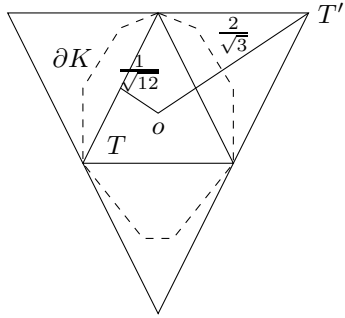
\begin{center}
\begin{tikzpicture}
\coordinate (X') at (2,0);
\coordinate (Y') at (4,4);
\coordinate (Z') at (0,4);

\coordinate (X) at ($(Y')!0.5!(Z')$);
\coordinate (XXY) at (1.5,3.8);
\coordinate (XYY) at (1,3);
\coordinate (Y) at ($(X')!0.5!(Z')$);
\coordinate (YYZ) at (1.8,1);
\coordinate (YZZ) at (2.2,1);
\coordinate (Z) at ($(X')!0.5!(Y')$);
\coordinate (ZZX) at (3,3);
\coordinate (ZXX) at (2.5,3.8);

\coordinate (P) at ($(XXY)!0.5!(XYY)$);
\coordinate (OO) at ($(2,2.6)$);

\draw ($(X)!0.5!(Y)$)--(2,2.6666666)node[midway, above] {$\frac{1}{\sqrt{12}}$} node[anchor=north] {$o$}-- node[midway, above] {$\frac{2}{\sqrt{3}}$} (Y');

\draw  (X')--(Y')--(Z')--cycle;
\draw (X)--(Y)--(Z)--cycle;
\draw [dashed] (X)--(XXY)--(XYY)-- (Y)--(YYZ)--(YZZ)--(Z)--(ZZX)--(ZXX)--cycle;

\draw ($(Y)!0.10!(Z)$) node[anchor=south west] {$T$};
\draw (Y') node[anchor=west] {$T'$};
\draw ($(XYY)!0.5!(XXY)$) node[anchor=east] {$\partial K$};
\end{tikzpicture}
\end{center}
\begin{obs}We make the following observations concerning lengths and areas.
\label{lengthsandareas}
\begin{itemize}
    \item We have $|T|=\frac{\sqrt{3}}{4}$, $|T'|=\sqrt{3}$, $|A|,|B|\in (0, \sqrt{3}]$ and $|K| \in [\frac{\sqrt{3}}{4},\sqrt{3}]$.
    \item For $p \in T'\setminus T$ we have $|op|\in [\frac{1}{\sqrt{12}},\frac{2}{\sqrt{3}}]$, and this in particular holds for $p \in \partial K$.
\end{itemize}
\end{obs}
\subsection{Definitions}
\label{defsubsection}
We now collect definitions we will use for the remainder of the paper.
\begin{defn}We define $$\gamma=t^2|\co(A)\setminus A|+(1-t)^2|\co(B)\setminus B|.$$
\end{defn}

\begin{defn}
\label{bisectingdef}
In a convex set $C$ containing $o$, given a point $p \in \partial C$ we say that $p$ is $(\theta,\ell)$-\textbf{bisecting} if the unique isoceles triangle $T_p(\theta,\ell)$ with angle $\theta$ at $p$ and equal sides $\ell$ such that $po$ internally bisects the corresponding angle is contained inside $C$.
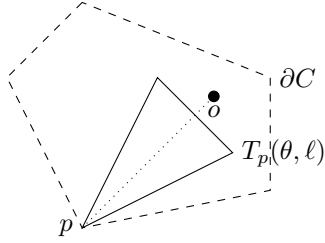
\begin{center}
\begin{tikzpicture}
\draw (1,2)--(0,0) node[anchor=east] {$p$}--(2,1) node[anchor=west] {$T_p(\theta,\ell)$}--cycle;
\filldraw (1.75,1.75) circle (2pt);
\draw (1.75,1.75) node[anchor=north] {$o$};
\draw[dotted] (0,0)--(1.75,1.75);
\draw[dashed](0,3)--(-1,2)--(0,0)--(2.5,0.5)--(2.5,2)  node[anchor=west] {$\partial C$}--cycle;
\end{tikzpicture}
\end{center}
\end{defn}

\begin{defn}
Given a convex set $C$, and a point $p \in \partial C$, we say that $p$ is $(\theta,\ell)$-\textbf{good} if there are any points $q,r \in C$ such that $|pq|,|pr| \ge \ell$ and $\angle pqr \ge 180^{\circ}-\theta$. Any point in $\partial C$ which is not $(\theta,\ell)$-good is $(\theta,\ell)$-\textbf{bad}.
\end{defn}

\begin{defn}
\label{pedefn}
Given a point $p$ and a set $E$ with $o\in \co(E)$, we denote $p_E$ the intersection of the ray $op$ with $\partial \co(E)$.
\end{defn}

\subsection{General Observations}
\label{genobssubsection}
\begin{obs}
\label{averagingargument}
Suppose we have subsets $R_A\subset \co(A),R_B \subset \co(B)$, and $z\in \mathbb{R}^2$. Let $H=H_{-\frac{1-t}{t},z}$ denote the negative homothety of ratio $-\frac{1-t}{t}$ through $z$. Then if $|R_A\cap H(R_B)| > t^{-2}\gamma$, or equivalently $|H^{-1}(R_A)\cap R_B| > (1-t)^{-2}\gamma$, then we have $z\in D_t$.\footnote{Note that $t^{-2}\gamma= |\co(A)\setminus A|+|\co(H(B))\setminus H(B)|$, so there is at least one $x \in R_A\cap H(R_B)\subset \co(A)\cap H(\co(B))$ which is not in $(\co(A)\setminus A)\cup (\co(H(B))\setminus H(B))$. Thus $x\in A \cap H(B)$, and $z=tx+(1-t)H^{-1}(x)\in D_t$.}
\end{obs}
\begin{obs}
\label{gammato0}
For sets $A,B$ with common volume $V$, Figalli and Jerison showed (see \Cref{FigJerBm}) that for fixed $\tau$ we have $|K \setminus A|V^{-1}$, $|K \setminus B|V^{-1} \to 0$ as $|D_t|V^{-1}\to 1$. In particular, as $V\in (0,\sqrt{3}]$ by \Cref{lengthsandareas}, we have
\begin{center}$|K\setminus A|$, $|K\setminus B|$, $|\co(A)\setminus A|,|\co(B)\setminus B|, \gamma \to 0$ as $d_{\tau}\to 0$.
\end{center}
\end{obs}

\subsection{Constants and their dependencies}
Fix $\tau$ and $\epsilon$. For the convenience of the reader, we describe roughly our choice of parameters throughout. First, we will take $M=1000 \in 2\mathbb{N}$ to be a universal constant and $\alpha=\frac{720^{\circ}}{M}$. Next, we will take $\xi$ such that $\epsilon \ge (\tau^2+(1-\tau)^2)(25\tau^{-1}M\xi^2+16000\tau^{-1}M\xi)$. Next, we take $\theta \le \frac{1}{2}^{\circ}$ such that $\frac{1}{2}\xi^2 \sin(28^{\circ})^6/\sin(4\theta) \ge 1$, and we take $\ell$ such that $\left(\frac{1440^{\circ}}{\theta}+3\right)4(1+100t^{-1})\ell\frac{100}{99}\sqrt{12}<\frac{1}{3}\alpha$. Finally, take $d_\tau$ sufficiently small to make various statements true along the way.

\section{Initial structural results} \label{GuaranteedRegionsSection}

In this section, we will show three preliminary propositions which quantify how close we may assume $A,B$ are to $K$, and how much of $\co(D_t)$ we can guarantee is covered by $D_t$ without resorting to a finer analysis of the boundaries of the various regions.
\begin{itemize}
    \item In \Cref{etascaleinside} we show that for any constant $\eta \in (0,1)$, if $d_{\tau}$ is sufficiently small in terms of $\eta$ then we have
$$(1-\eta)K\subset \co(A),\co(B),\co(D_t) \subset K.$$

\item In \Cref{perturb} we show that if $d_{\tau}$ is sufficiently small, then for every $z \in \partial K$ we have that $z,z_A,z_B,z_{D_t}$ are $(59^{\circ},\frac{1}{3})$-bisecting.
\item Finally, in \Cref{4rootg} we show that if $d_{\tau}$ is sufficiently small, then
$$(1-4t^{-1}\sqrt{\gamma})\co\left(D_t\right) \subset D_t.$$
\end{itemize}

\subsection{Showing  $\co(A),\co(B),\co(D_t)$ contain a large scaled copy of $K$}
\begin{prop}
\label{etascaleinside}
For any fixed $\eta \in (0,1)$, if $d_{\tau}$ is sufficiently small in terms of $\eta$, then $(1-\eta)K \subset \co(A),\co(B),\co(D_t)\subset K$.
\end{prop}

To prove \Cref{etascaleinside}, we need \Cref{bisectinglem} which guarantees that $\partial K$ behaves well under the notion of $(\theta,\ell)$-bisecting from \Cref{bisectingdef}.

\begin{lem}\label{bisectinglem}
Every point $p\in \partial K$ is $(60^{\circ},\frac{1}{2})$-bisecting.
\end{lem}

\begin{proof}
Note that the statement is trivially true if $p$ is a vertex of $\partial T$ (since then $T_p(60^{\circ},1)=T\subset K$), so assume otherwise. 
Let $x,y,z$ be the vertices of $T$ and $x'=-2x$, $y'=-2y$, $z'=-2z$ the corresponding vertices of $T'$. Let $p=p_z$ be in the triangle $xyz'$. Let $p_y\in xy'z$ and $p_x\in x'yz$ be the point $p_z$ rotated by $120^\circ$ and $240^\circ$ clockwise around $o$ respectively. Note that $p_xp_yp_z$ is an equilateral triangle with centre $o$, such that $\angle op_zp_y=30^\circ$. Let $p'$ be the intersection between segments $xz$ and $p_zp_y$.

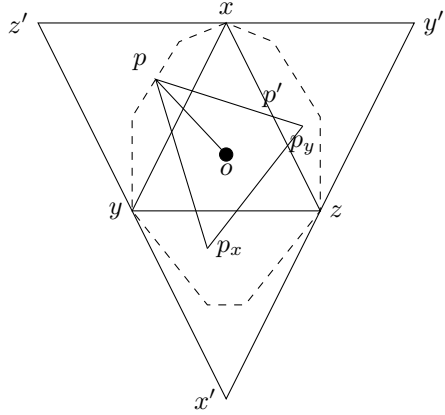
\begin{center}
\begin{tikzpicture}[scale=1.25]
\coordinate (X') at (2,0);
\coordinate (Y') at (4,4);
\coordinate (Z') at (0,4);

\coordinate (X) at ($(Y')!0.5!(Z')$);
\coordinate (XXY) at (1.5,3.8);
\coordinate (XYY) at (1,3);
\coordinate (Y) at ($(X')!0.5!(Z')$);
\coordinate (YYZ) at (1.8,1);
\coordinate (YZZ) at (2.2,1);
\coordinate (Z) at ($(X')!0.5!(Y')$);
\coordinate (ZZX) at (3,3);
\coordinate (ZXX) at (2.5,3.8);

\coordinate (P) at ($(XXY)!0.5!(XYY)$);
\coordinate (OO) at ($(2,2.6)$);

\draw  (X') node[anchor=east] {$x'$}--(Y') node[anchor=west] {$y'$}--(Z') node[anchor=east] {$z'$}--cycle;
\draw (X) node[anchor=south] {$x$}--(Y) node[anchor=east] {$y$}--(Z) node[anchor=west] {$z$}--cycle;
\draw [dashed] (X)--(XXY)--(XYY)-- (Y)--(YYZ)--(YZZ)--(Z)--(ZZX)--(ZXX)--cycle;

\coordinate (p') at ($(X)!0.5!(Z)$);
\coordinate (py) at ($(P)!1.25!(p')$);
\coordinate (pz) at ($(1.8,1.6)$);

\draw (OO) node[anchor=north] {$o$};
\filldraw (OO) circle (2pt);
\draw (OO)--(P) node[anchor=south east] {$p$}--(p') node[anchor=south] {$p'$}--(py) node[anchor=north] {$p_y$};
\draw (P)--(pz) node[anchor=west] {$p_x$};
\draw (py)--(pz);
\end{tikzpicture}
\end{center}


Note that $pp'\subset K$. We will show that $|pp'| \ge \frac{1}{2}$. Note that the points $o,p,p',x$ are concyclic as $\angle oxp'=30^{\circ}=\angle opp'$. We have $\angle pxp' \in [60^{\circ},120^{\circ}]$, so by the law of sines, $2r=\frac{|pp'|}{sin \angle pxp'} \le \frac{2}{\sqrt{3}}|pp'|$, where $r$ is the circumradius of this circle. But $2r \ge |ox|=\frac{1}{\sqrt{3}}$, so $|pp'| \ge \frac{1}{2}$. By showing a similar result for $p_zp_x$, we conclude that $T_{p}(60^{\circ},\frac{1}{2})$ lies in $K$.
\end{proof}

\begin{proof}[Proof of \Cref{etascaleinside}]
We prove this for $\co(A)$, the identical proof works for $\co(B)$ and then because $\co(D_t)=t\co(A)+(1-t)\co(B)$ we deduce the final containments. By \Cref{gammato0}, we can take $d_{\tau}$ sufficiently small in terms of $\eta$ so that $|K\setminus A|<\frac{\sqrt{3}}{36}\eta^2$. Let $p \in \partial K$, let $p'\in op$ be such that $|pp'|=\eta|op|$, and suppose for the sake of contradiction that $p' \not \in \co(A)$. Then as $|op|\in [\frac{1}{\sqrt{12}},\frac{2}{\sqrt{3}}]$ by \Cref{lengthsandareas}, we have $|pp'| \in [\frac{\eta}{\sqrt{12}},\frac{2\eta}{\sqrt{3}}]=[(\frac{2}{3}\eta)h, (\frac{8}{3}\eta) h]$ where $h=\frac{\sqrt{3}}{4}$ is the height of $T_p(60^{\circ},\frac{1}{2})$. A line separating $p$ from $\co(A)$ through $p'$ cuts off from  $T_p(60^{\circ},\frac{1}{2})$ an area of at least  $\min(\frac{1}{2},(\frac{2}{3}\eta)^2)|T_p(60^{\circ},\frac{1}{2})|=\frac{\sqrt{3}}{36}\eta^2$ on the $p$-side, which lies in $K\setminus A$, contradicting $|K\setminus A|<\frac{\sqrt{3}}{36}\eta^2$.
\end{proof}

\subsection{Showing points in $\partial K$, $\partial \co(A)$, $\partial \co(B)$, $\partial \co(D_t)$ are $(59^{\circ},\frac{1}{3})$-bisecting}
\begin{prop}
\label{perturb}\label{bisectingcor}
For $d_{\tau}$ sufficiently small, we have for every $z \in \partial K$ that $z,z_A,z_B,z_{D_t}$ are $(59^{\circ},\frac{1}{3})$-bisecting.
\end{prop}
\begin{proof}
By \Cref{etascaleinside} we can take $d_{\tau}$ sufficiently small so that $(1-\eta)K\subset \co(A),\co(B)\subset K$ with $\eta=10^{-9}$. Let $C$ be one of $K,\co(A),\co(B),\co(D_t)$. We have $T_z(60^{\circ},\frac{1}{2})\subset K$. Let $x,y$ denote the other two vertices of the triangle, and let $x'=(1-\eta)x$, $y'=(1-\eta)y$. Note that $x',y' \in (1-\eta)K\subset C$.
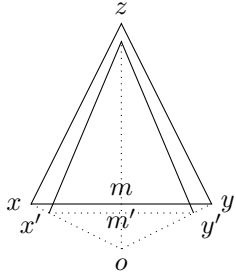
\begin{center}
\begin{tikzpicture}[scale=1.2]
\draw (-1,0) node[anchor=east] {$x$}--(0,2) node[anchor=south] {$z$}--(1,0) node[anchor=west] {$y$}--cycle;
\draw[dotted] (-1,0) node[anchor=north] {$x'$}--(0,-0.5)--(1,0) node[anchor=north] {$y'$};
\draw[dotted] ($(-1,0)!0.2!(0,-0.5)$)-- node[midway, label={[xshift=0cm, yshift=-0.45cm]$m'$}]{} ($(1,0)!0.2!(0,-0.5)$);
\draw (0,0) node[anchor=south]{$m$};

\draw ($(-1,0)!0.2!(0,-0.5)$)--(0,1.8)--($(1,0)!0.2!(0,-0.5)$);
\draw[dotted] (0,-0.5) node[anchor=north] {$o$}--(0,1.8); 
\end{tikzpicture}
\end{center}
Let $m$ be the midpoint of $xy$ and $m'$ be the midpoint of $x'y'$. Then $|x'm'|=\frac{1}{4}(1-\eta)$,  $|m'z_C|\le |mz_C|+|mm'|\le |mz|+\eta |om| \le \frac{\sqrt{3}}{4}+\eta \frac{2}{\sqrt{3}}$ by \Cref{lengthsandareas}, and similarly $|m'z_C| \ge |mz|-|zz_C|-|m'm|\ge |mz|-\eta (|oz|+|om|)\ge \frac{\sqrt{3}}{4}-2\eta\frac{2}{\sqrt{3}}$ (these are true even if $o$ is inside the triangle $xyz$). Thus, by inspecting the right triangles $x'm'z_C$ and $y'm'z_C$, because $\tan(29.5^{\circ})(\frac{\sqrt{3}}{4}+\eta\frac{2}{\sqrt{3}}) <\frac{1}{4}(1-\eta)$ and $\frac{1}{\cos(29.5^{\circ})}(\frac{\sqrt{3}}{4}-2\eta\frac{2}{\sqrt{3}}) > \frac{1}{3}$, the vertices of $T_{z_C}(59^{\circ},\frac{1}{3})$ lie in the triangle $x'y'z_C\subset C$. 
\end{proof}

\begin{cor}\label{tangentangle}
Let $C$ be $K,\co(A),\co(B)$ or $\co(D_t)$. For $d_{\tau}$ sufficiently small, given $z\in\partial C$ and a supporting line $l$ to $C$ at $z$, we have $\angle l, zo\in (29^\circ,180^{\circ}-29^{\circ})$.
\end{cor}

\subsection{Showing $D_t$ contains a large scaled copy of $\co(D_t)$}

\begin{prop}
\label{4rootg}
For $d_{\tau}$ sufficiently small, we have $$(1-4t^{-1}\sqrt{\gamma})\co(D_t) \subset D_t.$$
In particular, if $z\in \partial \co(D_t)$ and $p\in oz$ has $|pz|\ge 5t^{-1}\sqrt{\gamma}$, then $p \in D_t$.
\end{prop}

To show \Cref{4rootg}, we need the following lemma.

\begin{lem}
\label{assumplem}For every $\eta\in (0,1)$ and $d_{\tau}$ sufficiently small in terms of $\eta$, we have $(1-\eta)K \subset D_t$.
\end{lem}
\begin{proof}
We may assume that $\eta \le 10^{-9}$.
We take $d_\tau$ sufficiently small in terms of $\eta$ such that
$\frac{1-\eta}{1-\frac{\eta}{2}}K \subset \co(A)$ by \Cref{etascaleinside}, and $t^{-2}\gamma<\pi(\frac{1}{100}\eta)^2$ by \Cref{gammato0}. First, we show that for every $k\in K$ we have
$$B\left((1-\eta)k,\frac{1}{100}\eta\right) \subset \co(A),\co(B).$$

We show the $\co(A)$ containment, the other containment's proof is identical. 


Write $k=\lambda k'$ with $k' \in \partial K$ and $\lambda \in [0,1]$.
Because $k'$ is $(60^{\circ},\frac{1}{2})$-bisecting we see that
$$B\left(\left(1-\frac{\eta}{2}\right)k',\frac{\eta}{2\sqrt{12}}\sin(30^{\circ})\right)\subset T_{k'}(60^{\circ},\frac{1}{2})\subset K,$$ as $|ok'| \ge \frac{1}{\sqrt{12}}$ by \Cref{lengthsandareas}. Thus $$B\left((1-\eta)k',\frac{\eta}{20}\right)\subset B\left((1-\eta)k',\frac{1-\eta}{1-\frac{\eta}{2}}\frac{\eta}{2\sqrt{12}}\sin(30^{\circ})\right)\subset \left(\frac{1-\eta}{1-\frac{\eta}{2}}\right)K\subset \co(A),$$ and so
$B((1-\eta)k,\frac{\lambda}{20}\eta)\subset \co(A)$. If $\lambda \ge \frac{1}{5}$, then $B((1-\eta)k,\frac{1}{100}\eta)\subset \co(A)$, as desired.

Otherwise, assume $\lambda <\frac{1}{5}$. By \Cref{lengthsandareas} we have $|k'|\le \frac{2}{\sqrt{3}}$, so it follows that $|(1-\eta)\frac{100}{99}k|+ \frac{1}{99} \le \frac{1}{\sqrt{12}}$, the distance from $o$ to $\partial T$, and hence  $B((1-\eta)\frac{100}{99}k,\frac{1}{99})\subset T$. Hence, we have $B((1-\eta)k,\frac{1}{100})\subset \frac{99}{100}T\subset \co(A).$ Thus we always have
$B((1-\eta)k,\frac{1}{100}\eta)\subset \co(A)$ as desired. 

Let $k\in K$. To check that $k=t(1-\eta)k+(1-t)(1-\eta)k \in D_t$, in the notation of \Cref{averagingargument} we take $R_A=R_B=B((1-\eta)k,\frac{1}{100}\eta)\subset \co(A),\co(B)$. Then $|R_A \cap H_{-\frac{1-t}{t},z}(R_B)|=|R_A|=\pi (\frac{1}{100}\eta)^2>t^{-2}\gamma$. Hence, we conclude by \Cref{averagingargument} that $z\in D_t$.
\end{proof}

\begin{proof}[Proof of \Cref{4rootg}]
Let $\eta=10^{-9}$, and take $d_{\tau}$ sufficiently small so that \Cref{bisectingcor} and \Cref{assumplem} apply, and that $\gamma \le  \frac{t^2}{16}$ by \Cref{gammato0}. Let $z=tx+(1-t)y \in \partial \co(D_t)$ where $x \in \partial \co(A)$ and $y \in \partial \co(B)$.  We will show that $z'=(1-4\lambda t^{-1}\sqrt{\gamma})z$ lies in $D_t$ for all $\lambda \in [1,\frac{t}{4\sqrt{\gamma}}]$.


By \Cref{bisectingcor} we have $x,y$ are $(59^{\circ},\frac{1}{3})$-bisecting. Define $x',y'$ analogously to $z'$, and note that $tx'+(1-t)y'=z'$ and $|xx'|,|yy'|, |zz'|\in [\frac{4}{\sqrt{12}}\lambda t^{-1}\sqrt{\gamma},\frac{8}{\sqrt{3}}\lambda t^{-1}\sqrt{\gamma}]$, $|oz| \le \frac{2}{\sqrt{3}}$ by \Cref{lengthsandareas}. Because $\frac{1}{4}|xx'|, \frac{1}{4}|yy'| \le |zz'|$, if either $|xx'|$ or $|yy'|$ is at least $\frac{1}{100}$, then $|zz'|\ge \frac{1}{25}$, which by \Cref{assumplem}  implies 

$$z'\in \left(1-\frac{|zz'|}{|oz|}\right)K\subset \left(1-\frac{\sqrt{3}}{50}\right)K\subset (1-\eta)K\subset D_t.$$ 

Assume now that $|xx'|,|yy'|<\frac{1}{100}$, so that the altitudes from $x$ (resp. $y$) of $T_x(59^{\circ},\frac{1}{3})$ (resp. $T_y(59^{\circ},\frac{1}{3})$) exceed $2|xx'|$ (resp. $2|yy'|$). Because $\lambda \ge 1$ we have 
$$|xx'|,|yy'|\ge \frac{4\sqrt{\gamma}}{\sqrt{12}}\lambda t^{-1}\ge 1.001t^{-1}\sqrt{\frac{\gamma}{\pi}}/\sin(29.5^{\circ}).$$ 
Together the last two sentences show that  $B(x',1.001t^{-1}\sqrt{\frac{\gamma}{\pi}})\subset  T_x(59^{\circ},\frac{1}{3})\subset \co(A)$, and $B(y',1.001t^{-1}\sqrt{\frac{\gamma}{\pi}})\subset  T_y(59^{\circ},\frac{1}{3}) \subset \co(B)$. By applying \Cref{averagingargument} with $R_{A}= B(x',1.001t^{-1}\sqrt{\frac{\gamma}{\pi}})$ and $R_B=B(y',1.001t^{-1}\sqrt{\frac{\gamma}{\pi}})$, we conclude that $z'\in D_t$. 
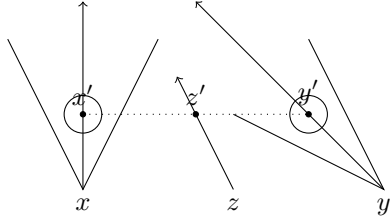
\begin{center}
\begin{tikzpicture}
\draw (-5,2) --(-4,0) node[anchor=north] {$x$}--(-3,2);
\draw (-1,2)--(0,0) node[anchor=north] {$y$}--(-2,1);
\draw (-4,1) node[anchor=south] {$x'$} circle (0.25);
\filldraw (-4,1) circle (1pt);
\draw (-1,1) node[anchor=south] {$y'$} circle (0.25);
\filldraw (-1,1) circle (1pt);
\draw[dotted] (-4,1)--(-1,1);
\draw[->] (-4,0)--(-4,2.5);
\draw[->] (0,0)--(-2.5,2.5);
\draw[->] (-2,0) node[anchor=north] {$z$}--(-2.75,1.5);
\filldraw (-2.5,1) node[anchor=south] {$z'$} circle (1pt);
\end{tikzpicture}
\end{center}
Finally, $|zz'|=4t^{-1}\sqrt{\gamma}|oz|\le \frac{8}{\sqrt{3}}t^{-1}\sqrt{\gamma} < |pz|$, so $p\in D_t$.
\end{proof}

\section{Decomposing $\partial \co(D_t)$ into good arcs, and bad arcs of small total angular size}
\label{decgoodbadsection}
Recall that $M\in 2\mathbb{N}$ be some universal constant ($1000$ suffices), and set $\alpha=\frac{720^{\circ}}{M}$.

\begin{defn}
For any $s$, we denote by $\bads$ the collection of arcs formed by the set of all points in $\partial \co(D_t)$ within Euclidean distance $s$ of a $(\theta,\ell)$-bad point (which is a union of arcs). We let $\goods$ denote the remaining arcs in $\partial \co(D_t)$, which we subdivide into arcs of angular length at most $\frac{1}{3}\alpha$ 
\end{defn}

\begin{prop}\label{partitionlem}For $d_\tau$ sufficiently small, there exists an increasing function $\ell=\ell(\theta)$ for $\theta<180^\circ$, such that the union of arcs $\bigcup \hbad$ has total angular size at most $\frac{1}{3}\alpha$.
\end{prop}



\begin{proof}
Take $d_\tau$ sufficiently small so that $\frac{99}{100}K\subset \co(D_t)$ by \Cref{etascaleinside}.

Choose a point on $\partial \co(D_t)$, and form a polygon $P$ inscribed in $\partial \co(D_t)$ by traveling around clockwise and picking the first vertex at distance $\ell$ from the previous vertex, all the way until the polygon would self-intersect, and then we simply join the first and last vertex with an edge. Then all sides are of length $\ell$ except one side of possibly smaller size. Moreover, each vertex of the polygon is within distance $\ell$ of every point of the next subtended arc of $\partial \co(D_t)$.

We let $\mathcal{S}^{\text{good}}$ be the collection of arcs of $\co(D_t)$ which arise as the arc subtended by $m_2m_3$, where $m_1,m_2,m_3,m_4$ are four consecutive vertices of the polygon $\partial \co(D_t)$, with $|m_1m_2|=|m_2m_3|=|m_3m_4|=\ell$ and $\angle m_1m_2m_3,\angle m_2m_3m_4 \ge 180^\circ-\frac{\theta}{2}$. We claim that every point $s\in \arc{q}\in \mathcal{S}^{\text{good}}$ is $(\theta,\ell)$-good. To see this, note that the angle condition in particular implies that $\angle m_1m_2m_3,\angle m_2m_3m_4>90^\circ$, so the rays $m_1m_2$ and $m_4m_3$ meet at a point $r$ as shown in the figure below. 
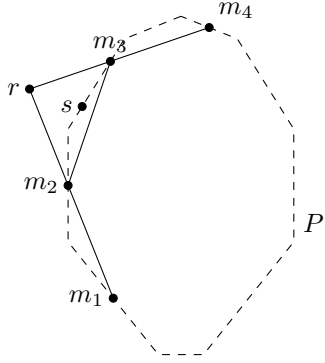
\begin{center}
    \begin{tikzpicture}[scale=1.5]
    \coordinate (X') at (2,0);
\coordinate (Y') at (4,4);
\coordinate (Z') at (0,4);
\coordinate (X) at ($(Y')!0.5!(Z')$);
\coordinate (XXY) at (1.5,3.8);
\coordinate (XYY) at (1,3);
\coordinate (Y) at ($(X')!0.5!(Z')$);
\coordinate (YYZ) at (1.8,1);
\coordinate (YZZ) at (2.2,1);
\coordinate (Z) at ($(X')!0.5!(Y')$);
\coordinate (ZZX) at (3,3);
\coordinate (ZXX) at (2.5,3.8);
\draw [dashed] (X)--(XXY)--(XYY)-- (Y)--(YYZ)--(YZZ)--(Z) node[anchor=south west]{$P$}--(ZZX)--(ZXX)--cycle;

\coordinate (M1) at ($(Y)!0.5!(YYZ)$);
\coordinate (M2) at ($(Y)!0.5!(XYY)$);
\coordinate (M3) at ($(XXY)!0.25!(XYY)$);
\coordinate (M4) at ($(X)!0.5!(ZXX)$);
\coordinate (S) at ($(XYY)!0.25!(XXY)$);

\tkzInterLL(M1,M2)(M3,M4) \tkzGetPoint{E}
\filldraw (M1) circle (1pt);
\filldraw (M2) circle (1pt);
\filldraw (M3) circle (1pt);
\filldraw (M4) circle (1pt);
\filldraw (E) circle (1pt);
\filldraw (S) circle (1pt);

\draw (E)--(M1) node[anchor=east] {$m_1$};
\draw (M2) node[anchor=east] {$m_2$}--(M3) node[anchor=south] {$m_3$};
\draw (E)--(M4) node[anchor=south west] {$m_4$};
\draw (E) node[anchor=east] {$r$};
\draw (S) node[anchor=east] {$s$};
    \end{tikzpicture}
\end{center}
We now show that $m_1,m_4$ realize $s$ as a $(\theta,\ell)$-good point. First, note that $|m_1s| \ge \ell=|m_1m_2|$ because $\angle m_1m_2s \ge 90^{\circ}$. Similarly $|m_4s| \ge \ell=|m_3m_4|$. Finally, $\angle m_1
sm_4 \ge \angle m_1rm_4 \ge 180^{\circ}-\theta$, where the first inequality follows as $s$ lies inside the triangle $m_1rm_4$, and the second as $\angle rm_2m_3,\angle rm_3m_2 \le \frac{\theta}{2}$.




Let $\mathcal{S}^{\text{bad}}$ be the collection of remaining arcs of $\partial \co(D_t)$ subtended by sides of $P$ which are not in $\mathcal{S}^{\text{good}}$. As the sum of the exterior angles of $P$ is $360^\circ$, the number of interior angles which are strictly less than $180^\circ-\frac{\theta}{2}$ is at most $\frac{720^\circ}{\theta}$.  Thus, $|\mathcal{S}^{\text{bad}}|\le \frac{1440^\circ}{\theta}+3$ (we add $3$ for the arc subtended by the last side of the polygon and the two adjacent arcs).  Note that every $(\theta,\ell)$-bad point is contained in an arc in $\mathcal{S}^{\text{bad}}$.

For each arc $\arc{q}\in\mathcal{S}^{\text{bad}}$ let $x_{\arc{q}}$ denote its clockwise starting point and $I_{\arc{q}}:=\partial \co(D_t) \cap B(x_{\arc{q}},(1+100t^{-1})\ell)$ the set of all points of  $\partial \co(D_t)$ within Euclidean distance at most $(1+100t^{-1})\ell$ of $x_{\arc{q}}$. This includes the points within Euclidean distance at most $100t^{-1}\ell$ of $\arc{q}$. Let $I:= \bigcup I_{\arc{q}} $, so that $  \bigcup \hbad \subset I$. 

Recall that $\frac{99}{100}K\subset\co(D_t)$, so that $\partial \co(D_t)\subset T'\setminus \frac{99}{100}T$ and thus $|ox|\geq \frac{99}{100}\frac{1}{\sqrt{12}}$ by \Cref{lengthsandareas}. The angular size of $B(x_{\arc{q}},(1+100t^{-1})\ell)$ and thus of $I_{\arc{q}}$ is at most  $2\sin^{-1}((1+100t^{-1})\ell)\frac{100}{99}\sqrt{12})\leq 4(1+100t^{-1})\ell\frac{100}{99}\sqrt{12}$.
We conclude that $ \bigcup \hbad \subset I$ has angular size at most $$\left(\frac{1440^{\circ}}{\theta}+3\right)4(1+100t^{-1})\ell\frac{100}{99}\sqrt{12},$$
which we can make smaller than $\frac13\alpha$ by choosing $\ell$ sufficiently small.

\end{proof}

\begin{defn}
We will always denote by $\ell=\ell(\theta)$ the increasing function of $\theta$ produced by the lemma above.
\end{defn}

\begin{obs}\label{2ldistance}
Every point in an arc in $\goods$ has distance at least $s$ to all $(\theta,\ell)$-bad points in $\partial \co(D_t)$, and we have the partition (up to a finite collection of endpoints)
$$\bigsqcup \goods \sqcup \bigsqcup \bads = \partial \co(D_t).$$
\end{obs}


\section{Replacing $5t^{-1}\sqrt{\gamma}$ with $\xi \sqrt{\gamma}$ on arcs in $\pgood$}
\label{xirootgammasection}



This section is devoted to proving the following proposition.
\begin{prop}\label{guarenteelem}
For every $\xi \in (0,1)$ there exists $\theta>0$, such that for $d_\tau$ sufficiently small in terms of $\xi$ the following is true. For every $p \in \arc{q} \in \pgood$ (recalling $\ell=\ell(\theta)$) and $p' \in op$ with $|pp'| \ge \xi \sqrt{\gamma}$, we have $p' \in D_t$.
\end{prop}
We outline the proof of \Cref{guarenteelem}. Suppose first that $p$ is the $t$-weighted average of points $x_A$ and $y_B$ which are distance at most $\ell$ apart. Then $x_{D_t}$, $y_{D_t}$ are both close enough to $p$ that by definition of $\pgood$, $x_{D_t}$ is $(\theta,\ell)$-good in $\co(A)$ and $y_{D_t}$ is $(\theta,\ell)$-good in $\co(B)$, which by \Cref{switchcor} implies $x_A,y_B$ are $(2\theta,\frac{\ell}{2})$-good, yielding certain angular regions at $x_A$ and $y_B$ lying in $\co(A)$ and $\co(B)$ respectively.

If instead the distance is at least $\ell$, then the triangles $ox_Ay_A$ and $oy_Bx_B$ serve as the large angular regions at $x_A$ and $y_B$ respectively.

In either case, the fact that $p \in \partial \co(D_t)$ implies the angular regions are in suitable directions so that \Cref{quadrilateral} applies, showing in either case these regions are suitable for an application of \Cref{averagingargument}, and we conclude.

\begin{lem}
\label{arcsinlem}
If we perturb the endpoints of a line segment of length $\ell$ each by an amount $r<\frac{\ell}{2}$, then the newly created line segment is rotated by at most $\sin^{-1}\frac{2r}{\ell}$.
\end{lem}
\begin{proof}
Consider two circles of radius $r$ around the two endpoints of the segment, then the maximally rotated segment is one of the interior bitangents to these circles.
\end{proof}

\begin{lem}
\label{distlinepoint}
In a triangle with vertices $a,b,c$, suppose that $\angle acb \in (28^{\circ},180^{\circ}-28^{\circ})$. Then the distance from $c$ to $ab$ is at least
$\sin(14^{\circ})\min(|ac|,|bc|).$
\end{lem}
\begin{proof}
Let $z$ be the foot of the perpendicular from $c$ to the line $ab$. We have either $\angle acz \le 90^{\circ}-14^{\circ}$ or $\angle bcz \le 90^{
\circ}-14^{\circ}$. Suppose without loss of generality that $\angle azc \le 90^{\circ}-14^{\circ}$. Then $|cz| = (\cos \angle azc)|ac| \ge \sin(14^{\circ})|ac|$.
\end{proof}

\begin{lem} \label{switchcor}
For $d_\tau$ sufficiently small in terms of $\theta$, if $x_{D_t}$ is $(\theta, \ell)$-good in $\co(D_t)$, then $x_A$ is $(2\theta,\ell/2)$-good in $\co(A)$ and $x_B$ is $(2\theta,\ell/2)$-good in $\co(B)$.
\end{lem}
\begin{proof}
We prove the statement for $x_A$, the statement for $x_B$ is proved identically.
Let $\eta = \frac{\sqrt{3}\ell}{8}\sin(\theta/2)$ (recall $\ell$ is defined to be a function of $\theta$), and take $d_\tau$  sufficiently small so that $(1-\eta)K\subset \co(A),\co(B),\co(D_t)\subset K$ by \Cref{etascaleinside}.
Let $y,z$ be the other two points in $\co(D_t)$ realizing $x_{D_t}$ as $(\theta,\ell)$-good. Because $(1-\eta)K \subset \co(A),\co(D_t) \subset K$, we have $|x_{D_t}x_A| \le \eta \frac{2}{\sqrt{3}}$. Defining $y'=(1-\eta)y \in \co(A)$ and $z'=(1-\eta)z \in \co(B)$ we have $|yy'|,|zz'| \le \eta \frac{2}{\sqrt{3}}$. Thus by \Cref{arcsinlem}, as $\sin^{-1}(\frac{4\eta}{\sqrt{3}\ell}) < \theta/2$  we have $\angle y'x_Az' \ge 180^{\circ}-2\theta$. As $|x_{D_t}x_A|+|yy'|\leq\frac{4\eta}{\sqrt{3}} < \frac{\ell}{2}$, by the triangle inequality $|x_Ay'| \ge \frac{\ell}{2}$. Similarly $|x_Az'| \ge \frac{\ell}{2}$, so we see that $y',z'$ realize $x_A$ as $(2\theta,\frac{\ell}{2})$-good.
\end{proof}

\begin{lem} \label{quadrilateral}
Let $m, n$ be two points and let $l_m^1,l_m^2$ and $l_n^1,l_n^2$ be pairs of rays originating at $m,n$, respectively and label $u,v,x,y$ as shown in the figure. Assume further that $\angle unv = \angle ymu\ge 28^{\circ} $. Denote $\angle num = \theta$ and $|mn| = r$. 
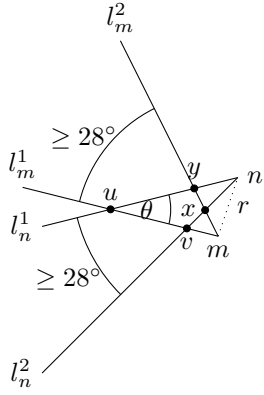
\begin{center}
\begin{tikzpicture}[scale=1.3]
\coordinate (x) at (-3,0);
\coordinate (y) at (3,0);
\coordinate (p) at (0,0);
\coordinate (ga) at (2,0.5);
\coordinate (oa) at (2,2);
\coordinate (p') at (0.1,0.3);
\coordinate (ob) at (-1,2);
\coordinate (gb) at (-2,0.5);
\coordinate (pr) at ($(p)!2!(p')$);
\coordinate (oar) at ($(oa)!2!(p')$);
\coordinate (gar) at ($(ga)!2!(p')$);

\draw[name path=l1n] (gar) node[anchor=east] {$l^1_n$}--(pr) node[anchor=west] {$n$};
\draw[name path=l2n] (pr)--(oar) node[anchor=east] {$l^2_n$};
\draw[name path=l2m] (p) node[anchor=north] {$m$}--(ob) node[anchor=south] {$l^2_m$};
\draw[name path=l1m] (p)--(gb) node[anchor=south] {$l^1_m$};

\path [name intersections={of = l1m and l1n}];
  \coordinate (u')  at (intersection-1);
  
  \path [name intersections={of = l1m and l2n}];
  \coordinate (v')  at (intersection-1);
  
    \path [name intersections={of = l2m and l2n}];
  \coordinate (x')  at (intersection-1);
  
      \path [name intersections={of = l2m and l1n}];
  \coordinate (y')  at (intersection-1);
  
\draw (u') node[anchor=south] {$u$};
\filldraw (u') circle (1pt);
\draw (v') node[anchor=north] {$v$};
\filldraw (v') circle (1pt);
\draw (x') node[anchor=east] {$x$};
\filldraw (x') circle (1pt);
\draw (y') node[anchor=south] {$y$};
\filldraw (y') circle (1pt);
\draw[dotted] (p)-- node[midway,right] {$r$} (pr);
\draw pic[draw,angle radius=0.8cm, "$\theta$"] {angle=p--u'--pr};
\draw pic[draw, angle radius=1.9cm, "$\ge 28^{\circ}$" shift={(-9mm,6mm)}] {angle=ob--p--gb};

\draw pic[draw, angle radius=2.2cm, "$\ge 28^{\circ}$" shift={(-11mm,-7mm)}] {angle=gar--pr--oar};
\end{tikzpicture}
\end{center}
Then we have the area lower bound $|uvxy| \ge \frac{1}{2} r^2 \sin(28^{\circ} )^6/ \sin (\theta).$
\end{lem}
\begin{proof}

First, we note that $$|uvxy| \ge |uvy| = |umn|\cdot \frac{|uv|}{|um|}\cdot \frac{|uy|}{|un|}.$$ By the law of sines, we have $|um|= r \sin(\angle unm)/ \sin(\theta)$ and $|un| = r \sin(\angle umn)/ \sin (\theta)$. We have $\angle unm,\angle umn \ge 28^{\circ}$, so as the sum of the angles of the triangle $umn$ is $180^{\circ}$, we have $\angle unm,\angle umn \in [28^{\circ},180^{\circ}-28^{\circ}]$. Therefore \begin{align*}|umn|&=\frac{1}{2}|um||un| \sin(\theta) = \frac{1}{2} r^2 \sin(\angle unm) \sin (\angle umn) / \sin(\theta) \\ &\ge \frac{1}{2} r^2\sin(28)^2 / \sin(\theta).\end{align*}
Next, we have \begin{align*}
\frac{|uv|}{|um|}= \frac{|unv|}{|unm|}= \frac{|nv|}{|nm|} \frac{\sin(\angle unv)}{\sin(\angle unm)}= \frac{\sin(\angle umn)\sin(\angle unv)}{\sin (\angle nvm)\sin(\angle unm)} \ge \sin (\angle umn) \sin ( \angle unv ) \ge \sin(28^{\circ})^2,
\end{align*}
and by a symmetric argument we have $\frac{|uy|}{|un|} \ge \sin(28^{\circ})^2$.
Multiplying the bounds, we obtain $|uvxy| \ge \frac{1}{2} r^2 \sin(28^{\circ})^6/ \sin (\theta)$ as desired. \end{proof}

\begin{proof}[Proof of \Cref{guarenteelem}]
We choose parameters as follows.
\begin{itemize}
    \item $\theta \le \frac{1}{2}^{\circ}$ such that $\frac{1}{2}\xi^2\sin(28^{\circ})^6/\sin(4\theta) \ge 1$ and $\ell=\ell(\theta)\le \frac{1}{2}$.
    \item Next, take $\eta=\frac{\sqrt{3}}{8}\ell\sin(\theta)$ (note with this choice of $\eta$ we have $(1-\eta)\frac{1}{\sqrt{12}} \ge \frac{1}{2}\ell$).
    \item Next, take $\gamma_0$ such that $5t^{-2}\sqrt{\gamma_0}\le \frac{\ell}{20}\sin(4\theta)$.
    \item Finally, take $d=d_{\tau}$ sufficiently small so that
    \begin{itemize}
        \item $\gamma \le \gamma_0$ by \Cref{gammato0}
        \item  $(1-\eta)K\subset \co(A),\co(B),\co(D_t)\subset K$ by \Cref{etascaleinside},
        \item $p' \in D_t$ if $|pp'| \ge 5t^{-1}\sqrt{\gamma_0}$ by \Cref{4rootg}
        \item \Cref{tangentangle} and \Cref{switchcor} apply.
    \end{itemize}
\end{itemize}
By our choice of $d_\tau$ we may assume that $|pp'|\in[\xi \sqrt{\gamma},5t^{-1}\sqrt{\gamma}]$. Write $p= tx_A+(1-t)y_B$, with $x_A\in \partial \co(A), y_B \in \partial \co(B)$. Construct
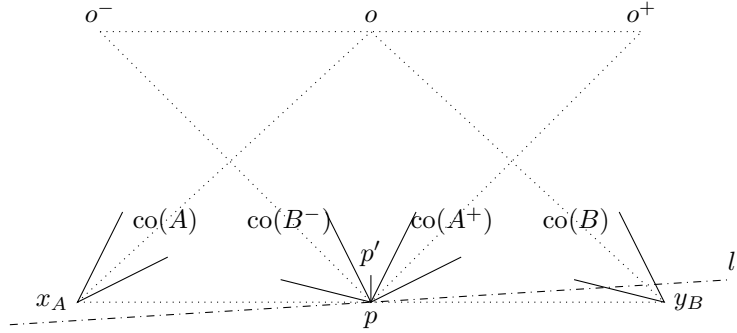
\begin{align*}
&A^{+} = A + \vec{x_Ap} &B^{-} = B+ \vec{y_Bp}\\
&o^{+}=o+\vec{x_Ap} &o^{-}=o+\vec{y_Bp}.
\end{align*}

\begin{center}
\begin{tikzpicture}[scale=1.2]
\draw[dotted] (x) node[anchor=east] {$x_A$}--(p) node[anchor=north] {$p$}--(y) node[anchor=west] {$y_B$};
\draw ($(x)+(0.5,1)$)--(x)--($(x)+(1,0.5)$);
\draw ($(y)+(-0.5,1)$)--(y)--($(y)+(-1,0.25)$);
\draw[dotted] (x)--(0,3) node[anchor=south] {$o$}--(y);
\draw ($(p)+(0.5,1)$)--(p)--($(p)+(1,0.5)$);
\draw ($(p)+(-0.5,1)$)--(p)--($(p)+(-1,0.25)$);
\draw[dotted] (-3,3) node[anchor=south] {$o^-$}--(3,3) node [anchor=south] {$o^+$};
\draw[dotted] (-3,3)--(p)--(3,3);
\draw ($(x)!0.3!(0,3)$) node {$\co(A)$};
\draw ($(y)!0.3!(0,3)$) node {$\co(B)$};
\draw ($(p)!0.3!(-3,3)$) node {$\co(B^-)$};
\draw ($(p)!0.3!(3,3)$) node {$\co(A^+)$};
\draw ($(p)!0.1!(0,3)$) node[anchor=south] {$p'$} -- (p);
\draw[dashdotted] (-4,-0.25)--(4,0.25) node [anchor=south] {$l$};
\end{tikzpicture}
\end{center}

Note that $o=to^{+}+(1-t)o^{-}$ and hence $p'$ is a point in triangle $ o^{+}po^{-}$ such that $|pp'|\in [\xi \sqrt{\gamma},5t^{-1}\sqrt{\gamma}]$. It is enough to show that for any such $p'$ we have $p' \in tA^{+}+(1-t)B^{-}$.

Because $p \in \partial \co(D_t)$, there is a supporting line $l$ at $p$ to $\co(D_t)$, and because $\co(D_t)$ is the Minkowski semisum $t\co(A)+(1-t)\co(B)$, this line also leaves $\co(A^{+}), \co(B^{-})$ on this same side as well. By Corollary \ref{tangentangle} we have that $\angle l,po^{+}, \angle l,po^{-} \in (29^{\circ}, 180^{\circ}-29^{\circ})$.

Our goal will be to produce points $g^+\in \co(A^+),g^-\in \co(B^-)$ with $|g^+p|,|g^-p| \ge \frac{\ell}{10}$, fitting into the following diagram
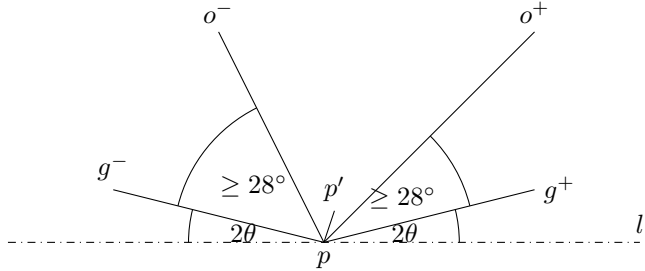
\begin{center}
\begin{tikzpicture}[scale=1.4]
\coordinate (x) at (-3,0);
\coordinate (y) at (3,0);
\coordinate (p) at (0,0);
\coordinate (ga) at (2,0.5);
\coordinate (oa) at (2,2);
\coordinate (p') at (0.1,0.3);
\coordinate (ob) at (-1,2);
\coordinate (gb) at (-2,0.5);
\draw (p) node[anchor=north] {$p$}--(ga) node[anchor=west] {$g^+$};
\draw (p)--(oa) node[anchor=south] {$o^+$};
\draw (p)--(p') node[anchor=south] {$p'$};
\draw (p)--(ob) node[anchor=south] {$o^-$};
\draw (p)--(gb) node[anchor=south] {$g^-$};
\draw[dashdotted] (x)--(y) node[anchor=south]{$l$};
\draw pic[draw, angle radius=1.8cm, "$2\theta$"] {angle=gb--p--x};
\draw pic[draw, angle radius=1.8cm, "$2\theta$"] {angle=y--p--ga};
\draw pic[draw, angle radius=2.0cm, "$\ge 28^{\circ}$"] {angle=ga--p--oa};
\draw pic[draw, angle radius=2.0cm, "$\ge 28^{\circ}$"] {angle=ob--p--gb};
\end{tikzpicture}
\end{center}
where the horizontal line is $l$, the points appear counterclockwise in the order $g^+,o^+,p',o^-,g^-$, and furthermore that $pg^+$ is rotated $2\theta$ counterclockwise from $\ell$ about $p$, $pg^-$ is rotated $2\theta$ clockwise from $\ell$ about $p$, and $\angle g^-po^-, \angle g^+po^+ \ge 28^{\circ}$.

\begin{claim}
\label{intermclaim}
If such points $g^+,g^-$ exist then $p'\in D_t$.
\end{claim}
\begin{proof}
Note that $|o^+p|=|ox_A|\ge (1-\eta)\frac{1}{\sqrt{12}}\ge\frac{\ell}{2}>\frac{\ell}{10}$ by \Cref{lengthsandareas}, and similarly $|o^-p| \ge \frac{\ell}{10}$. Furthermore, $|pp'| \le 5t^{-1}\sqrt{\gamma_0}\le \frac{\ell}{20}\sin(4\theta)$.

Let $S^-$ denote the triangle $g^-po^-$ and $S^+$ denote the triangle $g^+po^+$. Let $H$ denote the negative homothety $H=H_{p',-\frac{1-t}{t}}$ of ratio $-\frac{1-t}{t}$ at $p'$. Note that the inverse homothety $H^{-1}$ is a negative homothety with ratio $-\frac{t}{1-t}$ about $p'$.

First, we show that $$|H^{-1}(S^+)\cap S^-| \ge \frac{1}{2(1-t)^2}|pp'|^2\sin(28^{\circ})^6/\sin(4\theta).$$

This will be seen to follow from \Cref{quadrilateral}, applied with angle $4\theta$, $m=p$, $n=H^{-1}(p)$, $l_m^1=pg^-$, $l_m^2=po^-$, $l_n^1=H^{-1}(pg^+)$ and $l_n^2=H^{-1}(po^+)$. Let $u,v,x$ and $y$ be defined as in \Cref{quadrilateral} such that $\angle num=4\theta$. 

In order to apply \Cref{quadrilateral}, we need to check that the intersection of the triangles $H^{-1}(S^+)$ and $S^-$ contains the quadrilateral $uvxy$.

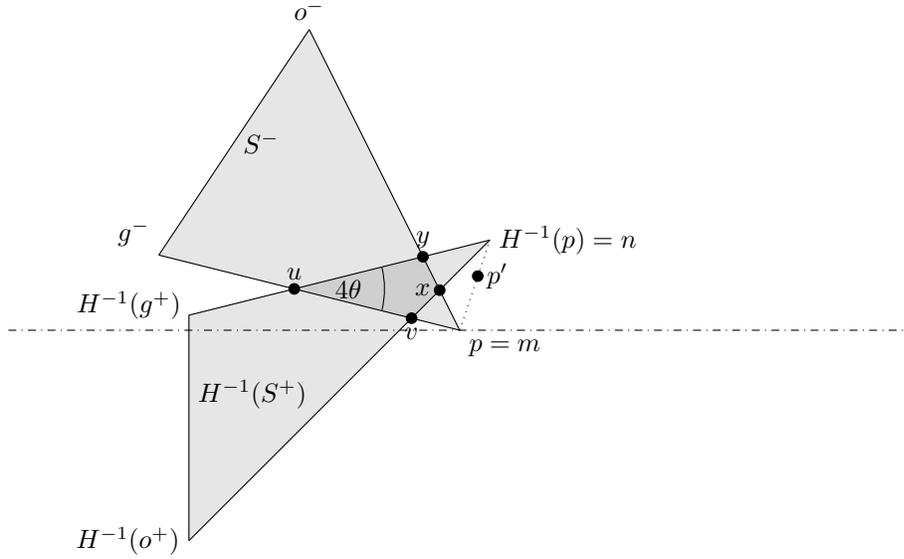
\begin{center}
\begin{tikzpicture}[scale=2]
\coordinate (x) at (-3,0);
\coordinate (y) at (3,0);
\coordinate (p) at (0,0);
\coordinate (ga) at (2,0.5);
\coordinate (oa) at (2,2);
\coordinate (p') at (0.1,0.3);
\coordinate (ob) at (-1,2);
\coordinate (gb) at (-2,0.5);
\coordinate (pr) at ($(p)!2!(p')$);
\coordinate (oar) at ($(oa)!2!(p')$);
\coordinate (gar) at ($(ga)!2!(p')$);

\draw[name path=l1n] (gar) node[label={[xshift=-0.8cm, yshift=-0.3cm]$H^{-1}(g^+)$}]{}--(pr) node[anchor=west] {$H^{-1}(p)=n$};
\draw[name path=l2n] (pr)--(oar) node[anchor=east] {$H^{-1}(o^+)$};
\draw[name path=l2m] (p) node[anchor=north west] {$p=m$}--(ob) node[anchor=south] {$o^-$};
\draw[name path=l1m] (p)--(gb) node[anchor=south east] {$g^-$};

\filldraw[color=black, opacity=0.1] (p)--(gb)--(ob)--cycle;
\filldraw[color=black, opacity=0.1] (pr)--(gar)--(oar)--cycle;
\draw (gb)--node[midway,right] {$S^-$} (ob);
\draw (oar)--($(oar)!0.65!(gar)$) node[anchor=west] {$H^{-1}(S^+)$} --(gar);

\path [name intersections={of = l1m and l1n}];
  \coordinate (u')  at (intersection-1);
  
  \path [name intersections={of = l1m and l2n}];
  \coordinate (v')  at (intersection-1);
  
    \path [name intersections={of = l2m and l2n}];
  \coordinate (x')  at (intersection-1);
  
      \path [name intersections={of = l2m and l1n}];
  \coordinate (y')  at (intersection-1);
  
\draw (u') node[anchor=south] {$u$};
\filldraw (u') circle (1pt);
\draw (v') node[anchor=north] {$v$};
\filldraw (v') circle (1pt);
\draw (x') node[anchor=east] {$x$};
\filldraw (x') circle (1pt);
\draw (y') node[anchor=south] {$y$};
\filldraw (y') circle (1pt);
\draw[dotted] (p)--  (pr);
\draw[dashdotted] (x)--(y);

\draw pic[draw,angle radius=1.2cm,"$4\theta$"] {angle=v'--u'--y'};


\coordinate (pppp') at ($(p)!1.2!(p')$);
\filldraw (pppp') circle (1pt);
\draw (pppp') node[anchor=west]{$p'$};

\end{tikzpicture}
\end{center}

Indeed, we have that $|un|=\sin(\angle upn) \frac{|mn|}{\sin(4\theta)} \le \frac{\ell}{20}\cdot \frac{t}{1-t}$, because $|mn|=\frac{1}{1-t}|pp'|\leq \frac{5}{t(1-t)}\sqrt{\gamma_0}\le \frac{\sin(4\theta)\ell}{20}\cdot \frac{t}{1-t}$, and similarly $|up| \le \frac{\ell}{20}\cdot \frac{t}{1-t}$. Then the triangle inequality shows that $|nv|,|py|\le \frac{\ell}{10}\cdot \frac{t}{1-t}$ as well, and we conclude from the fact that $|H^{-1}(o^+p)|,|H^{-1}(o^-p)|,|g^+p|,|g^-p| \ge \frac{\ell}{10}\cdot \frac{t}{1-t}$.

Next, because $|pp'|^2 \ge \xi^2 \gamma$, by our choice of $\theta_0$ this implies that $$|H^{-1}(S^+)\cap S^-| > (1-t)^{-2}\gamma.$$
Thus as
\begin{align*}
\frac{t^2}{(1-t)^2}|pg^+o^+\setminus A^+|+|pg^-o^-\setminus B^-|&\le \frac{t^2}{(1-t)^2}|\co(A^+)\setminus A^+|+|\co(B^-)\setminus B^-| \\&=\frac{t^2}{(1-t)^2}|\co(A)\setminus A|+|\co(B)\setminus B|\\
&=(1-t)^{-2}\gamma < |H^{-1}(S^+) \cap S^-|,
\end{align*}
a suitable modification of \Cref{averagingargument} shows $p'\in tA^++(1-t)B^-$ and hence $p'\in tA+(1-t)B$.
\end{proof}

Returning to the proof of the proposition, we note that exactly as in the start of \Cref{intermclaim} we have $|po^+|,|po^-| \ge \frac{\ell}{2}$. We now distinguish two cases.

\textbf{Case 1:} $|x_Ay_B| \ge \ell$. Recall the definitions of $x_B$ and $y_A$ from \Cref{pedefn}. By \Cref{lengthsandareas}, we have that $|x_Ax_B|, |y_Ay_B| \le \eta \frac{2}{\sqrt{3}}\le \frac{\ell}{4} $ and hence by the triangle inequality $ |x_Ay_A|, |x_By_B| \ge  \frac{\ell}{2}$.
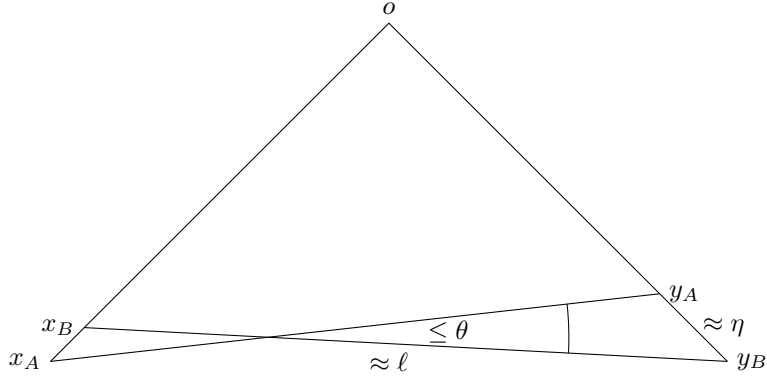
\begin{center}
\begin{tikzpicture}[scale=0.9]
\coordinate (ya') at (4,0);
\coordinate (yb') at (5,-1);
\draw (-5,-1) node [anchor=east]{$x_A$}--(0,4) node[anchor=south] {$o$}--(5,-1) node [anchor=west]{$y_B$};
\draw[name path=t1] (-4.5,-0.5) node[anchor=east] {$x_B$}--(5,-1);
\draw[name path=t2] (4,0) node [anchor=west] {$y_A$}--(-5,-1);
\draw (4.5,-0.5) node[anchor=west] {$\approx \eta$};
\draw (0,-1) node {$\approx \ell$};
\path [name intersections={of = t1 and t2}];
  \coordinate (i)  at (intersection-1);
\draw pic[draw, angle radius=4 cm, "$\le \theta$"] {angle=yb'--i--ya'};
\end{tikzpicture}
\end{center}

We also have $\angle x_Ay_A, x_By_B \le \sin^{-1}(\frac{8\eta}{\sqrt{3}\ell})\le \theta$ by \Cref{arcsinlem}.

Define $y_{A}^+:=y_A + \vec{x_Ap} \in A^+, x_{B}^-=x_B+\vec{y_Bp} \in B^-$. We have that $$|py_{A}^+|=|x_Ay_A|, \quad |px_{B}^-|=|x_By_B|,$$ and these are all $\ge \frac{\ell}{2}$ by the above discussion. Furthermore, $\angle y_{A}^+px_{B}^-=\angle x_Ay_A,y_Bx_B \ge \pi - \theta $, and the line $l$ through $p$ has $y_{A}^+, o^+, p', o^-, x_{B}^-$ on one side, appearing in this order counterclockwise above $l$. To see this, note that as $p$ lies on the segment $x_Ay_B$, $\vec{x_Ap}$ lies on the same side of the line $ox_A$ as $y_A$ does, so $o \not \in \angle y_A^+po_A^+$.  In particular, this implies that $\angle l, py_{A}^+,\angle l, px_{B}^- \le \theta$.

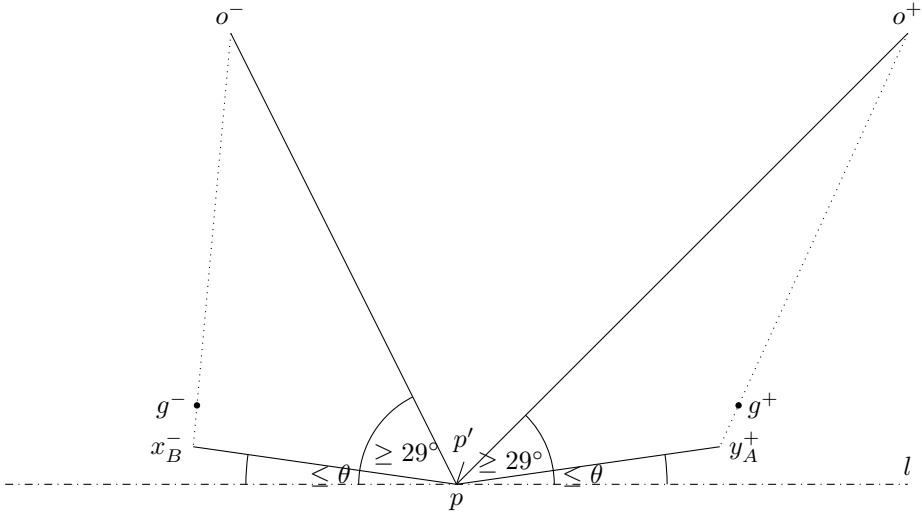
\begin{center}
\begin{tikzpicture}
\coordinate (x) at (-6,0);
\coordinate (y) at (6,0);
\coordinate (p) at (0,0);
\coordinate (ga) at (2,0.5);
\coordinate (oa) at (6,6);
\coordinate (p') at (0.1,0.3);
\coordinate (ob) at (-3,6);
\coordinate (gb) at (-2,0.5);
\coordinate (yaplus) at (3.5,0.5);
\coordinate (xbminus) at (-3.5,0.5);
\draw (p) node[anchor=north] {$p$}--(yaplus) node[anchor=west] {$y_A^+$};
\draw (p)--(oa) node[anchor=south] {$o^+$};
\draw (p)--(p') node[anchor=south] {$p'$};
\draw (p)--(ob) node[anchor=south] {$o^-$};
\draw (p)--(xbminus) node[anchor=east] {$x_B^-$};
\draw[dashdotted] (x)--(y) node[anchor=south]{$l$};
\draw pic[draw, angle radius=2.8cm, "$\le \theta$"] {angle=xbminus--p--x};
\draw pic[draw, angle radius=2.8cm, "$\le \theta$"] {angle=y--p--yaplus};
\draw pic[draw, angle radius=1.3cm, "$\ge 29^{\circ}$"] {angle=y--p--oa};
\draw pic[draw, angle radius=1.3cm, "$\ge 29^{\circ}$"] {angle=ob--p--x};
\draw[dotted] (oa)--(yaplus);
\draw[dotted] (ob)--(xbminus);
\filldraw ($(oa)!0.9!(yaplus)$) circle (1pt);
\filldraw ($(ob)!0.9!(xbminus)$) circle (1pt);
\draw ($(oa)!0.9!(yaplus)$) node[anchor=west]{$g^+$};
\draw ($(ob)!0.9!(xbminus)$) node[anchor=east]{$g^-$};
\end{tikzpicture}
\end{center}

Because $\angle l, po^+, \angle l, po^- \ge 29^{\circ}$ and $2\theta<29^{\circ}$, we have $\angle l, py_{A}^+\le 2\theta < \angle l,po^+$ and $\angle l, px_{B}^-\le 2\theta < \angle l,po^-$. These imply the existence of points $$g^+ \in y_{A}^+o^+ \subset \co(A^+),\text{ and }g^- \in x_{B}^-o^- \subset \co(B^-),$$ such that $\angle l, pg^+, \angle l, pg^- =2 \theta.$ Because $\angle l, py_A^+, \angle l, px_{B}^- \le \theta$ and $2\theta \le 1^{\circ}$, we have $$\angle g^+po^+, \angle g^-po^- \ge 29^{\circ}-2\theta\ge 28^{\circ}.$$ It is clear from the construction that $g^+, o^+, p', o^-, g^-$ also appear in this order counterclockwise above $l$. Finally, recall $|po^+|\ge \frac{\ell}{2}$, so by \Cref{distlinepoint} as $\angle o^+py_{A}^+ \in (28^{\circ},180^{\circ}-28^{\circ})$ we have $$|pg^+| \ge \min(|py_{A}^+|,|po^+|)\sin(14^{\circ}) \ge \frac{\ell}{10},$$ and similarly $|pg^-| \ge \frac{\ell}{10}$. 


\textbf{Case 2:} $|x_Ay_B| \le \ell$. Then $|x_Ap|,|y_Bp| \le \ell$, and we have $|x_{D_t}x_A|,|y_{D_t}y_A| \le \frac{2}{\sqrt{3}}\eta \le \frac{\ell}{4}$ by \Cref{lengthsandareas}. Thus by the triangle inequality $|x_{D_t}p|,|y_{D_t}p| \le \frac{5}{4}\ell<2\ell$. By definition of $\pgood$, since $p\in \arc{q} \in \pgood$ we have $x_{D_t},y_{D_t}$ are $(\theta,\ell)$-good. By \Cref{switchcor} we have that $x_A \in \co(A), y_B \in \co(B)$ are $(2\theta,\frac{\ell}{2})$-good. Therefore, there exists $$e_1, e_2 \in \co(A),\text{ and }f_1, f_2 \in \co(B)$$ such that $$\angle e_1x_Ae_2, \angle f_1y_Bf_2 \ge 180 - 2\theta,\text{ and }|e_1x_A|, |e_2x_A|, |f_1y_B|, |f_2y_B| \ge \frac{\ell}{2}.$$ Let \begin{align*}
    e_1^+=e_1+\vec{x_Ap},\qquad e_2^+=e_2+\vec{x_Ap}\\
    f_1^-=f_1+\vec{y_Bp},\qquad f_2^-=f_2+\vec{y_Bp}
\end{align*} such that $e_1^+,e_2^+\in \co(A^+)$ and $f_1^-,f_2^-\in \co(B^-)$.
With this notation we have that
$\angle e_1^+pe_2^+, \angle f_1^-pf_2^- \ge 180 - 2\theta$ and $|e_1^+p|, |e_2^+p|, |f_1^-p|, |f_2^-p| \ge \frac{\ell}{2}$. Recall that $\angle l, po^+, \angle l, po^- \in (29^{\circ}, 180^{\circ}-29^{\circ})$.

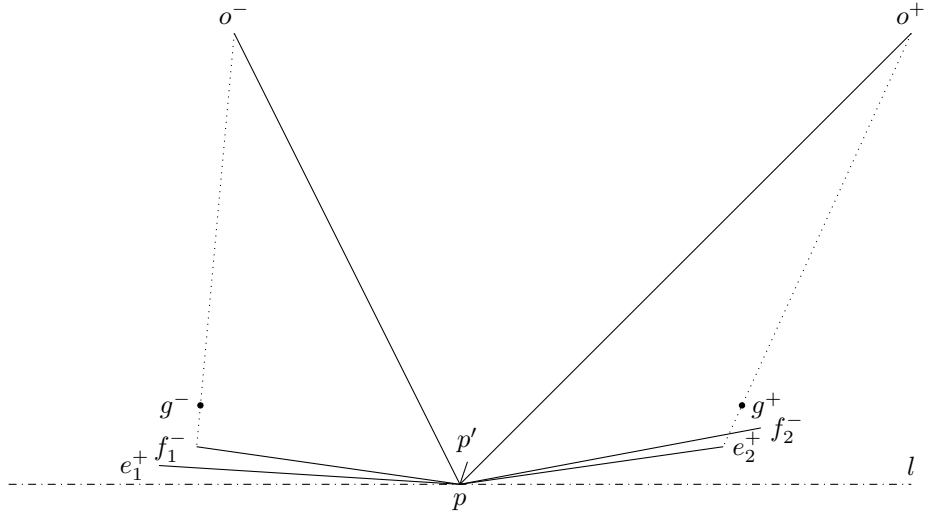
\begin{center}
\begin{tikzpicture}
\coordinate (x) at (-6,0);
\coordinate (y) at (6,0);
\coordinate (p) at (0,0);
\coordinate (ga) at (2,0.5);
\coordinate (oa) at (6,6);
\coordinate (p') at (0.1,0.3);
\coordinate (ob) at (-3,6);
\coordinate (gb) at (-2,0.5);
\coordinate (yaplus) at (3.5,0.5);
\coordinate (xbminus) at (-3.5,0.5);
\coordinate (e1plus) at (-4,0.25);
\coordinate (f2minus) at (4,0.75);
\draw (p) node[anchor=north] {$p$}--(yaplus) node[anchor=west] {$e_2^+$};
\draw (p)--(oa) node[anchor=south] {$o^+$};
\draw (p)--(p') node[anchor=south] {$p'$};
\draw (p)--(ob) node[anchor=south] {$o^-$};
\draw (p)--(xbminus) node[anchor=east] {$f_1^-$};
\draw[dashdotted] (x)--(y) node[anchor=south]{$l$};
\draw[dotted] (oa)--(yaplus);
\draw[dotted] (ob)--(xbminus);
\filldraw ($(oa)!0.9!(yaplus)$) circle (1pt);
\filldraw ($(ob)!0.9!(xbminus)$) circle (1pt);
\draw ($(oa)!0.9!(yaplus)$) node[anchor=west]{$g^+$};
\draw ($(ob)!0.9!(xbminus)$) node[anchor=east]{$g^-$};

\draw (p)--(e1plus);
\draw (e1plus) node[anchor=east] {$e_1^+$};

\draw (p)--(f2minus);
\draw (f2minus) node[anchor=west] {$f_2^-$};
\end{tikzpicture}
\end{center}

Notice that the line $l$ through $p$ leaves $e_1^+,e_2^+, f_1^-, f_2^- o^+, o^-, p' $ on one side, and that up to relabelling the points, $e_2^+,o^+,p',o^-,f_1^-$ appear in this order counterclockwise above $l$. Note that $\angle l,e_2^+p, \angle l, f_1^-p \le 2\theta $. Construct points \begin{align*}g^+ \in e_2^+o^+ \subset \co(A^+),\text{ and }g^- \in f_1^-o^- \subset \co(B^-)\end{align*} such that $\angle l, pg^+, \angle l, pg^- =2 \theta$ and note that $\angle g^+po^+, \angle g^-po^- \ge 28^{\circ}$ as $2\theta \le 1^{\circ}$. We can see from the construction that the points $g^+,o^+,p',o^-,g^-$ also appear in this order counterclockwise above $l$.  Finally, recall $|po^+|\ge \frac{\ell}{2}$, so by \Cref{distlinepoint} as $\angle o^+pe_2^+ \in (28^{\circ},180-28^{\circ})$, we have
$$|pg^+| \ge \min(|pe_2^+|,|po^+|)\sin(14^{\circ}) \ge \frac{\ell}{10},$$ and similarly that $|pg^-| \ge \frac{\ell}{10}$.
\end{proof}

\section{Covering $\partial \co(D_t)$ with parallelograms}
\label{paralelegramcoveringsection}
From now on, we let $\theta,\ell$ depend on $\xi \in (0,1)$ as in \Cref{guarenteelem}, and always assume that $d_{\tau}$ is sufficiently small so that \Cref{guarenteelem} holds. We will fix $\xi$ in terms of $\epsilon$, so when we say to take $d_\tau$ sufficiently small, we implicitly will take it sufficiently small in terms of our choice of $\xi$.

In this section, we construct a partition $\jj$ of $\partial \co(D_t)$ into small straight arcs $\arc{q}$, and parallelograms $R_{\arc{q}}$ which have one side on $\arc{q}$ such that $$ \co(D_t)\setminus D_t \subset \bigcup_{\arc{q}\in \jj} R_{\arc{q}}.$$

Recall that in \Cref{4rootg} we showed that for $d$ sufficiently small $D_t$ contains all points at radial distance $5t^{-1}\sqrt{\gamma}$ from $\partial \co(D_t)$. Furthermore, in \Cref{guarenteelem} we improved the bound to $\xi \sqrt{\gamma}$ for points in $\partial \co(D_t)$ that belong to arcs in $\pgood$.   

We will for the remainder of the paper be using $\goods,\bads$ exclusively for $s=2\ell,3\ell,100t^{-1}\ell$.
 Note that
$$\vbad \subset \bad \subset \hbad,$$
$$\pgood \supset \good \supset \rgood.$$
Thus \Cref{guarenteelem} also applies to points that belong to arcs in $\good$ and $\rgood$, and \Cref{partitionlem} also shows that the total angular size of arcs in $\vbad$ and $\bad$ is at most $\frac{1}{3}\alpha$.
We remark in what follows that we use
\begin{itemize}
    \item $\good \cup \bad$ to determine the heights of the $R_{\arc{q}}$, and
    \item $\rgood\cup \hbad$ to determine directions of the parallelograms $R_{\arc{q}}$.
\end{itemize}

\subsection{Definitions}
We first refine the partitions $\goods\cup \bads$ of $\partial co(D_t)$ for $s=2\ell,3\ell,100t^{-1}\ell$ into small straight segments.
\begin{defn}
Let $\jj$ be a partition of $\co(D_t)$ formed as a common refinement to all of the sets of arcs from the partitions $$\pgood\cup \vbad,\qquad \good \cup \bad,\qquad \rgood \cup \hbad$$ of $\co(D_t)$, into straight line segments of length at most $\xi\sqrt{\gamma}$. For $s\in \{2\ell,3\ell,100t^{-1}\ell\}$, define the partition $\jgoods\cup \jbads$ of $\jj$ by setting $\arc{q}\in \jgoods$ if and only if $\arc{q}\subset \arc{q}'\in \goods$.
\end{defn}

We will now in \Cref{vqdefn} choose the vectors $v_{\arc{q}}$ for $\arc{q}\in \jj$ with direction vectors $\widehat{v_{\arc{q}}}$ determined by the partition $\hbad\cup \rgood$, and with lengths determined by $\bad\cup\good$. We then in \Cref{Rqdefn} form parallelograms $R_{\arc{q}}$ with sides $\arc{q}$ and $v_{\arc{q}}$.


\begin{defn}
\label{vqdefn}
For an arc $\arc{q}\in \jj$, we define a vector $v_{\arc{q}}$ as follows.
\begin{itemize}
    \item We choose the direction vector $\widehat{v_{\arc{q}}}$ of $v_{\arc{q}}$ as follows. Let $\arc{q}\subset \arc{q}'\in \hbad \cup \rgood$. If $\arc{q}'$ is contained inside an angular interval $[m\alpha,(m+1)\alpha]$, we take the direction vector $\widehat{v_{\arc q}}$ to be the inward pointing direction at angle $(m+\frac{1}{2})\alpha$. Otherwise (recalling that $\arc{q}'\in \hbad \cup \rgood$ has angular length at most $\frac{1}{3}\alpha$) $\arc{q}'$ overlaps a unique angle $m\alpha$, and we take $\widehat{v_{\arc q}}$ to be the inward pointing vector at angle $m\alpha$.
    \item We choose the length of $v_{\arc{q}}$ to be
$$
||v_{\arc{q}}||=\begin{cases}
    15\sqrt{\gamma} &  \arc{q}\in \jbad\text{, and}\\
   3\xi \sqrt{\gamma} & \arc{q}\in \jgood.
\end{cases}
$$
\end{itemize}
For $p\in\partial \co(D_t)$, we denote $v_p=v_{\arc q}$, where $p \in \arc{q}\in\jj$.
\end{defn}
\begin{defn}
\label{Rqdefn}
For $\arc{q}\in\jj$, let $R_\arc{q}$ be a parallelogram with one side $\arc{q}$ and one side $v_\arc{q}$.
\end{defn}
By construction, the directions of each of the $v_{p}$ for $p\in \partial \co(D_t)$ are one of  $M=\frac{4\pi}{\alpha}$ directions, and the directions of the vectors are constant on arcs of $\co(D_t)$ from $\hbad \cup \rgood$. 

\begin{obs}\label{halfalpha}
For every point $p\in \partial \co(D_t)$ we have $\angle po,v_p<\frac{1}{2}\alpha$. 
\end{obs}

\subsection{Covering $\partial \co(D_t)$ with parallelograms}
Now are able to state the main result of this section.
\begin{prop}\label{coveringloss}
 For $d_\tau$ sufficiently small, we have $$\co(D_t)\setminus D_t \subset \bigcup_{\arc{q}\in \jj} R_{\arc{q}}.$$
\end{prop}

We need the following observation about the unit direction vectors $\widehat{v_{\arc{q}}}$ of $v_{\arc{q}}$.

\begin{lem}\label{directionobs}
Let $p\in \partial \co (D_t) $, and $p' \in op$. Then there exists $r \in \partial \co (D_t)$, with $\widehat{v_p}=\widehat{v_{r}}$ and this is parallel to $rp'$.  
\end{lem}
\begin{proof}
Let $z$ be the unique point on $\partial \co(D_t)$ with $zo$ in the direction of $\widehat{v_p}$.
By \Cref{halfalpha}, the angle between $\widehat{v_z}$ and $zo$ (which is in the direction $\widehat{v_p}$) is strictly less than $\frac{1}{2}\alpha$. As the $\widehat{v}$ angles occur in multiples of $\frac{1}{2}\alpha$, this implies $\widehat{v_z}=\widehat{v_p}$.

Let $r$ be the unique point on $\partial co(D_t)$ with $rp'$ in the direction of $v_p$. Then $r$ lies on the arc $pz$, so $\widehat{v_p}=\widehat{v_r}$ is parallel to $rp'$.
\end{proof}

\begin{proof}[Proof of \Cref{coveringloss}]

Assume that $d_{\tau}$ is sufficiently small so that \Cref{bisectingcor} and \Cref{4rootg} are true. Given a point $p\in\partial \co(D_t)$, define the interval $$S_p(\theta,\ell;\xi)=pp'$$ where $p' \in op$ is such that
$$|pp'|=\begin{cases}5\sqrt{\gamma}&p\in\arc{q}\subset \vbad\text{, and}\\ \xi \sqrt{\gamma} & p\in\arc{q}\subset \pgood.\end{cases}$$
By \Cref{4rootg} and \Cref{guarenteelem} we have $(\co(D_t)\setminus D_t)\cap op \subset S_p(\theta,\ell,\xi)$ for all $p \in \partial co(D_t)$. Thus denoting by $$\Lambda(\theta,\ell;\xi):=\bigcup_{p\in\partial \co(D_t)} S_p(\theta,\ell;\xi),$$ we have
$$\co(D_t)\setminus D_t \subset R(\theta,\ell;\xi).$$

Fix a point $p\in \partial \co(D_t)$, and let $p'\in S_p(\theta,\ell;\xi)=op\cap \Lambda(\theta,\ell;\xi)$. It suffices to show that $$p'\in \bigcup_{\arc{q}\in \jj}R_{\arc{q}}.$$ Note that by \Cref{directionobs} there exists a point $r' \in \partial co(D_t)$ such that $r'p'$ is parallel to $\widehat{v_{r'}}=\widehat{v_{p}}$. Let $r$ be the intersection of the line extended from the segment $\arc{q}$ and the ray $p'r'$. 

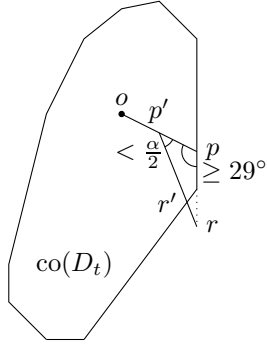
\begin{center}
\begin{tikzpicture}
\draw (4,0)--(3.5,0.5)--(3.5,1)-- (4,3) --(4.5,4)--(5,4.4)--(5.5,4.5)--(6,4)--(6,2);
\draw[name path=p1] (6,2)--(4.5,0);
\draw (4.5,0)--(4,0);

\coordinate (r') at (6,1.5);
\coordinate (p) at (6,2.5);
\coordinate (o) at (5,3);
\coordinate (p') at ($(o)!0.5!(p)$);
\draw (5,3) node [anchor=south] {$o$}--(6,2.5) node[anchor=west] {$p$};
\draw[dotted](6,2)--(r') node[anchor=west] {$r$};
\draw[name path=p2] (r')--(p');
\path [name intersections={of = p1 and p2}];
  \coordinate (r)  at (intersection-1);
\draw (r) node[anchor=east] {$r'$};
\draw pic[draw, angle radius=0.2 cm, "$\ge 29^{\circ}$" shift={(6mm,-2mm)}] {angle=o--p--r'};
\draw pic[draw, angle radius=0.2 cm, "$<\frac{\alpha}{2}$" shift={(-3.5mm,-2mm)}] {angle=r'--p'--p};

\draw (p') node[anchor=south] {$p'$};

\draw (4.375,1) node {$\co(D_t)$};

\filldraw (5,3) circle (1pt);
\end{tikzpicture}
\end{center}

Note that $\angle rpp'\in (29^{\circ},180^{\circ}-29^{\circ})$ by \Cref{bisectingcor}, and $\angle pp'r<\frac{1}{2}\alpha$ by \Cref{halfalpha}, so $\angle prp'\in (29^{\circ}-\frac{1}{2}\alpha,180^{\circ}-29^{\circ})$. Thus by the law of sines $|r'p'|\le |rp'|=\frac{\sin(\angle rpp')}{\sin(\angle prp')}|pp'|\leq 3 |pp'|$. 

If $\arc{q}\in\jpgood$, then $|pp'|\leq \xi\sqrt{\gamma}$, so $|r'p'|\leq 3\xi\sqrt{\gamma}$, and letting $r'\in\arc{r}\in \jj$ we have $p'\in R_{\arc{r}}\subset \bigcup_{\arc{q}\in\jj}R_\arc{q}$.

Alternatively if $\arc{q}\in\jvbad$ then $|pp'| \leq 5 \sqrt{\gamma}$. Note that $|pr'|\leq |pp'|+|rp'|\leq 4 |pp'|\leq \ell$, so $r'$ is in an arc $\arc{r}\in\jbad$. Hence, $|r'p'|\leq 15 \sqrt{\gamma}$, implying $p'\in R_{\arc{r}}\subset\bigcup_{\arc{q}\in\jj}R_\arc{q}$.

\end{proof}

\section{Preimages of the $R_{\arc{q}}$ associated to $A$ and $B$.}
\label{preimagesABsection}
By \Cref{coveringloss}, for $d_{\tau}$ sufficiently small we have
$$\co(D_t)\setminus D_t \subset \bigcup_{\arc{q}\in \jj} (R_{\arc{q}}\setminus D_t).$$
The boundary of $\co(D_t)$ is composed of translates of edges from $\partial \co(A)$ scaled by a factor of $t$ and of edges from $\partial \co(B)$ scaled by a factor of $(1-t)$.
\begin{defn}
Let $\jj=\mathcal{A}\sqcup \mathcal{B}$ be the partition defined as follows. For every arc $\arc{q}\in \jj$ (which is straight by construction), we let $\arc{q} \in \mathcal{A}$ if $\arc{q}$ is on a translated $t$-scaled edge from $\partial \co(A)$, and we let $\arc{q}\in \mathcal{B}$ if $\arc{q}$ is on a translated $(1-t)$-scaled edge from $\partial \co(B)$.
\end{defn}

\begin{defn}
For $\arc{q}\in \mathcal{A}$, let $p_{\arc{q},B}\in \partial \co(B)$ and $R_{\arc{q},A}\subset \mathbb{R}^2$ be the parallelogram with edge $\arc{q}_A\subset \partial \co(A)$ such that $$R_{\arc{q}}=tR_{\arc{q},A}+(1-t)p_{\arc{q},B}.$$ Similarly, for $\arc{q}\in \mathcal{B}$, let $p_{\arc{q},A}\in \partial \co(A)$ and $R_{\arc{q},B} \subset \mathbb{R}^2$ be the parallelogram with edge $\arc{q}_B\subset \partial \co(B)$ such that
$$R_{\arc{q}}=tp_{\arc{q},A}+(1-t)R_{\arc{q},B}.$$
\end{defn}
\begin{rmk}
The parallelogram $R_{\arc{q},A}$ (resp. $R_{\arc{q},B}$) may not be entirely contained inside $\co(A)$ (resp. $\co(B)$), and the various $R_{\arc{q},A}$ with $\arc{q} \in \mathcal{A}$ (respectively $R_{\arc{q},B}$ with $\arc{q}\in \mathcal{B}$) may not be disjoint.
\end{rmk}
\begin{prop}\label{containedloss} For $d_{\tau}$ sufficiently small, we have
$$\left|\co(D_t)\setminus D_t\right|\leq t^2\sum_{\arc{q}\in \mathcal{A}}|R_{\arc{q},A}\setminus A|+(1-t)^2\sum_{\arc{q}\in \mathcal{B}}|R_{\arc{q},B}\setminus B|$$
\end{prop}
\begin{proof}
Assume $d_{\tau}$ is sufficiently small that \Cref{coveringloss} holds. Then we have
$$\co(D_t)\setminus D_t\subset \bigcup_{\arc{q}\in \jj}(R_{\arc{q}}\setminus D_t).$$
The result then follows from the fact that if
\begin{itemize}
    \item $\arc{q}\in \mathcal{A}$ then $|R_\arc{q}\setminus D_t|\leq |R_\arc{q}\setminus (tA+(1-t)p_{\arc{q},B})|=t^2|R_{\arc{q},A}\setminus A|$, and if
    \item $\arc{q}\in \mathcal{B}$ then $|R_\arc{q}\setminus D_t|\le |R_{\arc{q}}\setminus (tp_{\arc{q},A}+(1-t)B)|= (1-t)^2|R_{\arc{q},B}\setminus B|$.
\end{itemize} 
\end{proof}

 From \Cref{containedloss}, we see that if the preimages in $A,B$ of these regions were disjoint and contained in $\co(A)$ and $\co(B)$, then we'd immediately obtain $|\co(D_t)\setminus D_t| \le t^2|\co(A)\setminus A|+(1-t)^2|\co(B)\setminus B|$.

Our goal will be to remove certain $R_{\arc{q},A}$ and $R_{\arc{q},B}$ to ensure that all remaining parallelograms are disjoint and are entirely contained in $\co(A)$ and $\co(B)$, such that the total area of the $R_{\arc{q},A}$ with $\arc{q}\in \mathcal{A}$ that were removed is at most $\epsilon |\co(A)\setminus A|$, and the total area of the $R_{\arc{q},B}$ with $\arc{q}\in \mathcal{B}$ that were removed is at most $\epsilon |\co(B)\setminus B|$. This will imply \Cref{mainthm'}.

\section{Far away weighted averages in $\partial \co(D_t)$ lie in $\jgood$}
\label{farweightedaveragessection}
We now show that points on the $\partial \co(D_t)$ which are the $t$-weighted average of points from $\partial \co(A)$, $\partial \co(B)$ that are at distance at least $20t^{-1}\ell$ lie in arcs from $\jgood$.

The main application will be to show that for parallelograms $R_{\arc{q}}$ with $\arc{q}\in \jbad$, we know that the point and parallelogram or parallelogram and point in $co(A)$ and $co(B)$ whose $t$-weighted average gives $R_{\arc{q}}$ are close to each other.
\begin{prop}\label{farisgoodlem}
For $d_\tau$ sufficiently small, if $p\in \partial \co(D_t)$ with $p=tx_A+(1-t)y_B$, where $x_A\in\partial \co(A)$, $y_B\in\partial \co(B)$ and $|x_Ay_B|\geq 20t^{-1}\ell$, then $p\in\arc{q}\in \jgood$.
\end{prop}
\begin{proof}
Let $\eta=\frac{40}{\sqrt{3}}\sin(\frac{\theta}{4})$. Assume $d_{\tau}$ is sufficiently small so that \Cref{tangentangle} holds, and $(1-\eta)K\subset \co(A),\co(B),\co(D_t)\subset K$ by \Cref{etascaleinside}. We will first show that $x_{D_t}$ and $y_{D_t}$ realize $p$ as a $(\frac{1}{2}\theta,19\ell)$-good point.
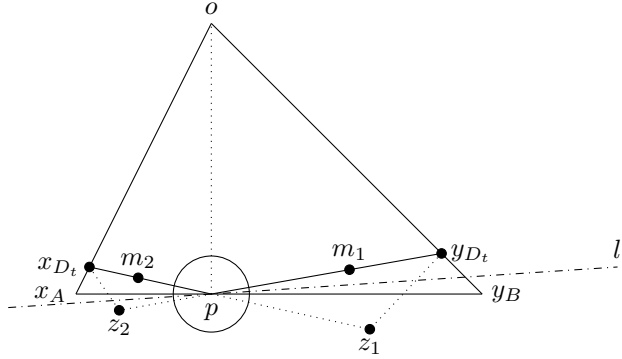
\begin{center}
\begin{tikzpicture}[scale=1.8]
\coordinate (xa) at (-1,0);
\coordinate (yb) at (2,0);
\coordinate (ooo) at (0,2);
\coordinate (pp) at (0,0);
\coordinate (xdt) at ($(xa)!0.10!(ooo)$);
\coordinate (ydt) at ($(yb)!0.15!(ooo)$);
\coordinate (z2) at ($(xdt)!2.3!(pp)$);
\coordinate (z1) at ($(ydt)!1.4!(pp)$);
\coordinate (m1) at ($(ydt)!0.4!(pp)$);
\coordinate (m2) at ($(xdt)!0.4!(pp)$);
\draw (xa)--(yb);
\draw (ooo) node[anchor=south] {$o$}--(xdt) node[anchor=east] {$x_{D_t}$}--(xa) node[anchor=east] {$x_A$};
\draw (ooo)--(ydt) node[anchor=west] {$y_{D_t}$}--(yb) node[anchor=west] {$y_B$};
\filldraw (xdt) circle (1pt);
\filldraw (ydt) circle (1pt);
\draw (xdt)-- (pp) node[anchor=north] {$p$} -- (ydt);
\draw[dotted] (pp)--(ooo);
\draw (pp) circle (8pt);
\draw[dashdotted] (-1.5,-0.1)--(3,0.2) node[anchor=south] {$l$};
\filldraw (z1) circle (1pt);
\filldraw (z2) circle (1pt);
\filldraw (m1) circle (1pt);
\filldraw (m2) circle (1pt);
\draw[dotted] (xdt)--(z1) node[anchor=north] {$z_2$}--(pp)--(z2) node[anchor=north] {$z_1$}--(ydt);
\draw (m1) node[anchor=south] {$m_1$};
\draw (m2) node[anchor=south] {$m_2$};
\end{tikzpicture}
\end{center}
 For the angle, note that by \Cref{lengthsandareas} we have $\angle x_{D_t}px_A \le \sin^{-1}(\frac{|x_Ax_{D_t}|}{|x_Ap|})\le \sin^{-1}(\frac{\eta|ox_A|}{20t^{-1}\ell})\le \sin^{-1}(\frac{\eta\sqrt{3}}{40t^{-1}\ell})\le \frac{\theta}{4}$ and similarly $\angle y_{D_t}py_B \le \sin^{-1}(\frac{|y_By_{D_t}|}{|y_Bp|})\le \frac{\theta}{4}$. For the lengths, notice that $|x_{D_t}x_A| \le \ell$ and similarly $|y_{D_t}y_A| \le \ell$, so by triangle inequality we have
 \begin{align*}
 |px_{D_t}| &\ge |px_A|-|x_{D_t}x_A| =(1-t)|x_Ay_B|-|x_{D_t}x_A|\ge 20\ell-\ell=19\ell,\text{ and}\\ |py_{D_t}| &\ge |py_B|-|y_{D_t}y_B|=t|x_Ay_B|-|y_{D_t}y_A|\ge 20\ell-\ell=19\ell.
 \end{align*}

Now, we show that $p\in \arc{q}\in \jgood$ by showing that if $p'\in \partial co(D_t)$ and $|pp'| \le 3\ell$, then we have $p'$ is $(\theta,\ell)$-good. Denote by $l$ the supporting line to $\co(D_t)$ at $p$, and note by \Cref{tangentangle} that $\angle l, op \in (29^{\circ},180^{\circ}-29^{\circ})$. The line $l$ intersects either the interior of the angle $\angle x_{D_t}px_A$ or $\angle y_{D_t}py_B$, so as we have already shown that $\angle x_{D_t}px_A,\angle y_{D_t}py_B \le \frac{\theta}{4}$, we have that $x_Ay_B$ makes an angle of at most $\frac{\theta}{4}$ with $l$. In particular, $\angle opx_A,\angle opy_B \in (29^{\circ}-\frac{\theta}{4},180^{\circ}-29^{\circ}+\frac{\theta}{4})\subset (28^{\circ},180^{\circ}-28^{\circ})$. Thus we may apply \Cref{distlinepoint} to triangles $x_Apo$ and $y_Bpo$ to conclude that the distance from $p$ to the lines $ox_A$ and $oy_B$ is at least $\sin(14^{\circ})\min(|px_A|,|po|,|py_B|) \ge \sin(14^{\circ})20\ell > 3\ell$. Because $ox_{D_t}py_{D_t} \subset \co(D_t)$, we conclude that $p'$ lies outside of the angle $x_{D_t}py_{D_t}$ (and because $p'\in \co(D_t)$, it lies on the same side of $l$ as $x_{D_t},y_{D_t}$). 

Let $z_1$ be in the ray $x_{D_t}p$ extended past $p$ such that $|z_1p|=|z_1y_{D_t}|$. Note that as $pz_1y_{D_t}$ is isosceles, $\angle pz_1y_{D_t}\geq \pi-\theta$, and note that $\angle y_{D_t}pz_1 \le \frac{\theta}{2}$. Analogously let $z_2$ be the point at $py_{D_t}$ which has $|z_2x_{D_t}|=|z_2p|$, so that $\angle pz_2x_{D_t}\geq \pi-\theta$ and $\angle x_{D_t}pz_2\le \frac{\theta}{2}$. Finally, let $m_1$ be the midpoint of $py_{D_t}$, and let $m_2$ be the midpoint of $px_{D_t}$, so that $\angle pm_1z_1=\angle pm_2z_2=90^{\circ}$.

We claim that $p' \in pm_1z_1 \cup pm_2z_2$. First, note that by the above we have shown that $p'$ lies in either the angular region $\angle m_1pz_1$ or $\angle m_2pz_2$. Thus as $pm_1z_1, pm_2z_2$ are right triangles, it suffices to note that $|pm_1|,|pm_2| \ge \frac{19}{2}\ell > 3\ell$.
Therefore, $p' \in pm_1z_1 \cup pm_2z_2 \subset py_{D_t}z_1 \cup px_{D_t}z_2$. Hence, $\angle y_{D_t}p'x_{D_t}\geq \pi-\theta$ and $p'$ is $(\theta, \ell)$-good since $|p'x_{D_t}|,|p'y_{D_t}| \ge 19\ell-3\ell > \ell$ by the triangle inequality. 



\end{proof}
\section{Bound on parallelograms jutting out of $\co(A),\co(B)$}
\label{juttingsection}
We will now show that the $R_{\arc{q},A}$ and $R_{\arc{q},B}$ which are not entirely contained in $\co(A)$ and $\co(B)$ have negligible total area.
\begin{prop}\label{notcontained}
For $d_\tau$ sufficiently small, we have
\begin{align*}\sum_{\arc{q}\in \mathcal{A}\text{ and }R_{\arc{q},A}\not\subset \co(A)}\big|R_{\arc{q},A}\big|\leq 25t^{-1}M\xi^2\gamma,\text{ and }
\sum_{\arc{q}\in \mathcal{B}\text{ and }R_{\arc{q},B}\not\subset \co(B)}\big|R_{\arc{q},B}\big|\leq 25t^{-1}M\xi^2\gamma.
\end{align*}
\end{prop}
To prove this proposition, we first use \Cref{farisgoodlem} to show that for such parallelograms we have $\arc{q} \in \jgood$.

\begin{lem}\label{stickingout}
For $d_\tau$ sufficiently small, if $\arc{q}\in \mathcal{A}$ and $R_{\arc{q},A}\not\subset \co(A)$ or $\arc{q}\in \mathcal{B}$ and $R_{\arc{q},B} \not \subset \co(B)$, then $\arc{q}\in \jgood$.
\end{lem}
\begin{proof}
The cases $\arc{q}\in \mathcal{A}$ and $\arc{q}\in \mathcal{B}$ are proved identically, so we will now suppose that $\arc{q}\in \mathcal{A}$. Assume $d_{\tau}$ is sufficiently small so that \Cref{bisectingcor} and \Cref{farisgoodlem} are true.
Recall that we defined $p_{\arc{q},B}\in \partial \co(B)$ and $\arc{q}_A \subset \co(A)$ such that $\arc{q}=t\arc{q}_A+(1-t)p_{\arc{q},B}$.

We first show that there exists a point $p_A\in \arc{q}_A$ such that $\angle p_Ao,v_{\arc{q}} \ge 29^{\circ}$. Indeed, by \Cref{bisectingcor} we know that every point in $x \in \arc{q}_A$ is $(59^\circ, \frac13)$-bisecting in $\co(A)$.
For $x\in \arc{q}_A$, let $x'=x+t^{-1}v_{\arc{q}}$, which lies on the opposite side of $\partial R_{\arc{q},A}$ to $p_A$. Note that $|xx'|\leq \frac{1}{10}$, so if $\angle ox,v_q\leq 29^\circ$, then $xx' \subset T_{p_A}(58^\circ,\frac13)$. Hence, as $R_{\arc{q},A}=\bigcup_{x \in \arc{q}_A}xx'\not\subset co(A)$ but $\bigcup_{x\in \arc{q}_A} T_x(58^\circ,\frac13)\subset \co(A)$, we find a point $p_A\in\arc{q}_A$ with $\angle p_Ao,v_{\arc{q}}\geq 29^\circ$.


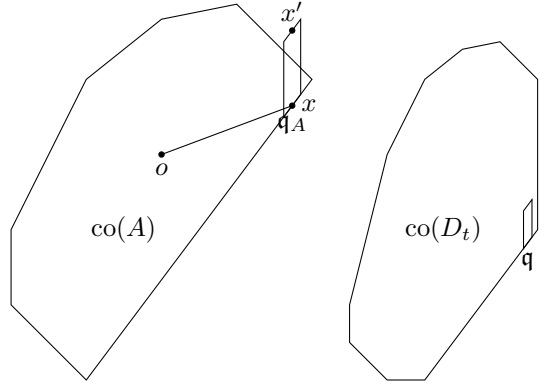
\begin{center}
\begin{tikzpicture}
\draw (0,0)--(-1,1)--(-1,2) --(0,4) --(1,4.8)--(2,5)--(3,4)--cycle;

\draw (4,0)--(3.5,0.5)--(3.5,1)-- (4,3) --(4.5,4)--(5,4.4)--(5.5,4.5)--(6,4)--(6,2)--(4.5,0)--cycle;

\coordinate (xx) at (2.7375,3.65);
\coordinate (xx') at (2.7375,4.65);

\draw (2.625,3.5)-- node[midway,below] {$\arc{q}_A$}(2.85,3.8) --(2.85,4.8)--(2.625,4.5)--cycle;

\filldraw (xx) circle (1pt);
\filldraw (xx') circle (1pt);
\draw (xx) node[anchor=west] {$x$};
\draw (xx') node[anchor=south] {$x'$};

\draw (5.8125,1.75)-- node[midway,below] {$\arc{q}$}(5.925,1.9)--(5.925,2.4)--(5.8125,2.25)--cycle;


\filldraw (1,3) circle (1pt);
\draw (1,3) node[anchor=north] {$o$}--(xx);
\draw (4.75,2) node{$\co(D_t)$};




\draw (0.5,2) node {$\co(A)$};

\end{tikzpicture}
\end{center}

Let $z=tp_A+(1-t)p_{\arc{q},B}\in \arc{q}$. By \Cref{halfalpha}, $\angle zo, v_{\arc{q}} \le \frac{1}{2}\alpha$. Hence $\angle p_Aoz \ge 29^{\circ}-\frac{1}{2}\alpha \ge 28^{\circ}$, so $|p_Az| \ge \sin(28^{\circ})|oz|>\frac{1}{100}$, so as $z$ lies on the segment $p_Ap_{\arc{q},B}$, we have $|p_Ap_{\arc{q},B}| > \frac{1}{100}$.
Therefore, by \Cref{farisgoodlem} applied with $x_A=p_A$ and $y_B=p_{\arc{q},B}$, we have $z\in \arc{q}\in \jgood$.

\end{proof}

We now know that parallelograms $R_{\arc{q},A}$ and $R_{\arc{q},B}$ which escape $\co(A)$ and $\co(B)$ have small height, since they are supported on arcs from $\jgood$. By showing that such arcs with a constant direction $v_p$ have small total length, we will obtain the following (recalling $M$ is the number of distinct $v_p$).

\begin{proof}[Proof of \Cref{notcontained}]
The proof below works for the $\co(B)$ inequality verbatim, so we focus on proving the $\co(A)$ inequality. Take $d_\tau$ sufficiently small so that \Cref{bisectingcor} holds, and so that $t^{-1}3\xi\sqrt{\gamma} \le \frac{1}{4}\sin(1^{\circ})$ by \Cref{gammato0}.

By \Cref{stickingout}, all $\arc{q}\in \mathcal{A}$ with $R_{\arc{q},A}\not\subset co(A)$ are in $\jgood$. Fix one of the $\le M$ vectors $v$ with $|v|=3\xi\sqrt{\gamma}$. It suffices to show
$$\sum_{\arc{q}\in \mathcal{A}\text{, }v_{\arc{q}}=v\text{, and } R_{\arc{q},A}\not\subset \co(A)}|R_{\arc{q},A}| \le 25t^{-1}\xi^2 \gamma.$$

Recall that by construction $v$ was chosen so that it was not parallel to any edge of $\co(A)$. Let $l,l'$ be the two lines in the direction $v$ which are tangent to $\co(A)$, and let $y$ and $y'$ be the points of contact with $\co(A)$. Note that every line in the direction $v$ between $y$ and $y'$ intersects each of the arcs $\partial \co(A)\setminus\{y,y'\}$ exactly once. As $\co(A)$ is convex, the cross-sectional slices in the $v$-direction satisfy unimodality. Hence there are exactly two pairs $(x_1,x_2)$ and $(x_1',x_2')$ of points in the two different arcs of $\partial \co(A)\setminus\{y,y'\}$ such that $x_1x_2=x_1'x_2'=t^{-1}v$ --- we let $(x_1,x_2)$ be the pair closer to $y$.

We will show that the lengths of the two minor arcs in $\co(A)$ between $x_1x_2$ and between $x_1'x_2'$ are both of length at most $24t^{-1}\sqrt{\gamma}$. We show this for $x_1x_2$ as the other case will be identical.

Note that $T_y(56^{\circ},\frac{1}{4})\subset T_y(59^{\circ},\frac{1}{3})\subset \co(A)$.  Let $z\in oy$ such that $|yz|=t^{-1}3\xi \sqrt{\gamma} \le \frac{1}{4}\sin(1^{\circ})$ and denote by $z_1,z_2$ the intersections of the extensions of the arms of $T_y(56^{\circ},\frac{1}{4})$ with the line through $z$ with direction vector $v$. We will show that the line $x_1x_2$ is closer to $y$ than the line $z_1z_2$ by showing that $|z_1z_2| \ge |x_1x_2|$ and applying unimodality.

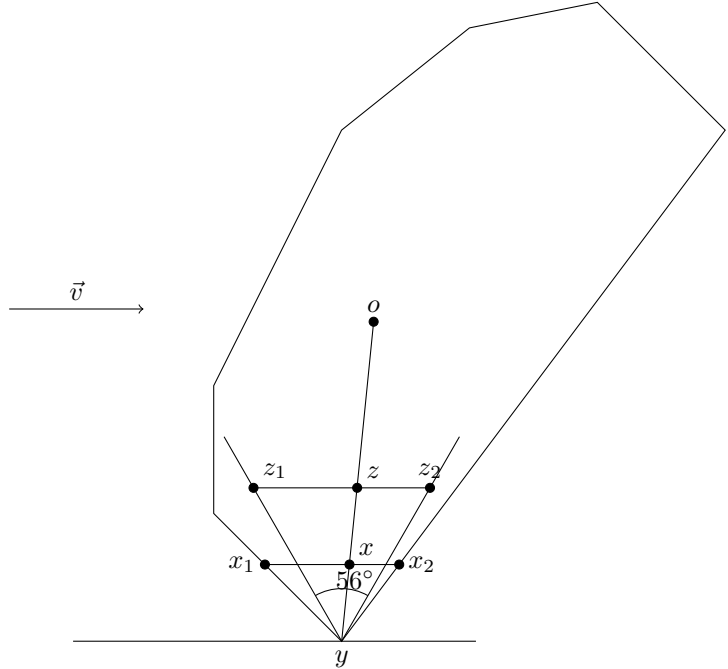
\begin{center}
\begin{tikzpicture}[scale=1.7]
\draw (4,0)--(3,1)--(3,2) --(4,4) --(5,4.8)--(6,5)--(7,4)--cycle;

\coordinate (r') at (6,1.5);
\coordinate (p) at (6,2.5);
\coordinate (o) at (4.25,2.5);
  \coordinate (r)  at (intersection-1);

\coordinate (x) at ($(4,0)!0.15!(7,4)$);
\coordinate (x') at ($(5.15,2)$);
\coordinate (y) at ($(4,0)!0.0!(4.5,0)$);
\coordinate (z) at ($(4,0)!0.6!(3,1)$);
\coordinate (z') at ($(2.85,2)$);
\coordinate (x'') at ($(x')!0.2!(y)$);
\coordinate (z'') at ($(z')!0.2!(y)$);
\coordinate (x''') at ($(x')!0.4!(y)$);
\coordinate (z''') at ($(z')!0.4!(y)$);
\filldraw (x) circle (1pt);
\filldraw (z) circle (1pt);
\draw (x) node[anchor=west] {$x_2$};
\draw (z) node[anchor=east] {$x_1$};
\draw (z'')--(y) node[anchor=north] {$y$}--(x'');
\draw[name path=x1x2] (x)-- (z);
\draw pic[draw, angle radius=0.7 cm, "$56^{\circ}$" shift={(1.8mm,4mm)}] {angle=x'--y--z'};

\filldraw (x''') circle (1pt);
\filldraw (z''') circle (1pt);
\filldraw (o) circle (1pt);
\draw (x''') node[anchor=south] {$z_2$};
\draw (z''') node[anchor=south west] {$z_1$};


\draw[<-] ($(x)-(2,-2)$)-- node[midway, above] {$\vec{v}$} ($(z)-(2,-2)$);

\draw[name path=supp1] ($(y)+(z)+(z)-(x)-(x)$)--($(y)+(x)-(z)$);
\draw[name path=oy]
(o) node[anchor=south] {$o$}--(y);
\draw[name path=z1z2]
(x''')--(z''');
\path [name intersections={of = oy and z1z2}];
  \coordinate (zzzz)  at (intersection-1);
\path [name intersections={of = oy and x1x2}];
  \coordinate (xxxx)  at (intersection-1);
\filldraw (xxxx) circle (1pt);
\filldraw (zzzz) circle (1pt);
\draw (xxxx) node[anchor=south west] {$x$};
\draw (zzzz) node[anchor=south west] {$z$};
\end{tikzpicture}
\end{center}

Note that $\angle z_1yz =28^{\circ}$ and $\angle z_1zy \in (29^{\circ}, 180^{\circ}-29^{\circ})$. Hence $\angle yz_1z \in (1^{\circ},180^{\circ}-57^{\circ})$ so $\sin \angle yz_1z \ge \sin(1^{\circ})$. Thus by the law of sines,
$$|yz_1| =\frac{\sin \angle z_1zy}{\sin \angle yz_1z} |yz| \le \frac{|yz|}{\sin 1^{\circ}}\le \frac{1}{4}.$$
Hence $z_1\in T_y(56^{\circ},\frac{1}{4})$ and by a similar argument we obtain $z_2 \in T_y(56^{\circ},\frac{1}{4})$. 

Now, $$|z_1z_2| \ge |z_1z|=\frac{\sin 28^{\circ}}{\sin \angle yz_1z}|yz|\ge \sin(28^{\circ})|yz|=t^{-1}3\xi \sqrt{\gamma}=|x_1x_2|.$$
Thus by the unimodality, the line $x_1x_2$ is closer than the line $z_1z_2$ to $y$, so denoting by $x=oy \cap x_1x_2$ we have $x$ lies in the segment $yz$. Hence
$$|yx| \le |yz|=t^{-1}3\xi\sqrt{\gamma}.$$ 

Note that there are up to $2$ arcs $\arc{q}_A$ which contain one of the points $x_1,x_1'$, and as each arc in $\jj$ has length at most $\xi\sqrt{\gamma}$ by construction, the total length of these arcs is at most $2t^{-1}\xi \sqrt{\gamma}$. 

If $v_{\arc{q}}=v$ and $R_{\arc{q},A} \not \subset \co(A)$, then $\arc{q}_A$ is contained in the arc of $\partial \co(A)\setminus \{y,y'\}$ containing $x_1,x_1'$, and $\arc{q}_A$ intersects either the minor arc subtended by $x_1y$ or by $x_1'y'$. Indeed, let $\widetilde{l}$ be the supporting line of $\arc{q}$. Then for any point $p\in \arc{q}$, by \Cref{bisectingcor} the angle $\angle po, \widetilde{l} \in (29^{\circ},180^{\circ}-29^{\circ})$, and by \Cref{halfalpha} $\angle po,v_{\arc{q}} \le \frac{\alpha}{2}$. Hence $v_{\arc{q}}$ lies on the same side of $\widetilde{l}$ as $\co(D_t)$. Therefore $v_{\arc{q}}$ lies on the same side of the supporting line $\widetilde{l}_A$ to $\arc{q}_A$ as $\co(A)$, so $\arc{q}_A$ lies in the arc of $\co(A)\setminus \{y,y'\}$ that contains $x_1,x_1'$. Now, if $\arc{q}_A$ does not intersect the minor arcs $x_1y$ or $x_1'y'$, then by unimodality, the $v$ cross-sectional lengths of $\co(A)$ on the arc $\arc{q}_A$ exceed $3\xi t^{-1}\sqrt{\gamma}=||t^{-1}v||$, which implies $R_{\arc{q}_A}$ is contained inside $\co(A)$.

Hence, the total width (measured in the direction $v^\perp$) of such parallelograms $R_{\arc{q},A}$ in direction $v$ which are not contained in $co(A)$ is at most $2\cdot t^{-1}3\xi\sqrt{\gamma}+2t^{-1}\xi \sqrt{\gamma}=8t^{-1}\xi \sqrt{\gamma} $.


Because all of the arcs $\arc{q}$ we are considering lie in $\jgood$, the total area of such parallelograms is then at most
$$(8t^{-1}\xi\sqrt{\gamma})(3\xi \sqrt{\gamma})=24t^{-1}\xi^2 \gamma.$$

\end{proof}

\section{Bounding overlapping parallelograms}
\label{boundingoverlappingsection}
We will now show that the $R_{\arc{q},A}$ and $R_{\arc{q},B}$ which we remove to guarantee non-overlapping have negligible area.

\begin{prop}\label{disjointbad}
For $d_\tau$ sufficiently small, if $\arc{q},\arc{q}'\in \jbad\cap \mathcal{A}$, then $|R_{\arc{q},A}\cap R_{\arc{q}',A}|=0$, and if $\arc{q,q}'\in \jbad\cap \mathcal{B}$, then $|R_{\arc{q},B}\cap R_{\arc{q}',B}|=0$.
\end{prop}
Because of \Cref{disjointbad}, it will suffice to bound overlaps between parallelograms supported on arcs in $\jgood$ with all other parallelograms.
\begin{prop}\label{overlapping}
For $d_\tau$ sufficiently small, we have $$\sum_{\arc{q}\in\jgood\cap \mathcal{A}\text{ and } \exists  \arc{q}'\in \mathcal{A}\setminus \{\arc{q}\}\text{ with } |R_{\arc{q},A}\cap R_{\arc{q}',A}|>0}\big|R_{\arc{q},A}\big| \leq 16000 t^{-1} M\xi \gamma $$
and similarly with $B$ and $\mathcal{B}$.
\end{prop}

\begin{proof}[Proof of \Cref{disjointbad}]  The proof we give works verbatim for $B$ and $\mathcal{B}$, so we focus on the case with $A$ and $\mathcal{A}$. We take $d_\tau$ sufficiently small such that \Cref{farisgoodlem} holds, and such that $\sqrt{\gamma} \le \ell$ by \Cref{gammato0}. Because $\arc{q},\arc{q}' \in \jbad$, we have $||v_{\arc{q}}||=||v_{\arc{q}'}||=15\sqrt{\gamma}$.
Consider the arcs $\arc{r,r}'\in\hbad$ such that $\arc{q}\subset \arc{r}$ and $\arc{q}'\subset \arc{r}'$. If $\arc{r}=\arc{r}'$ then $v_\arc{q}=v_{\arc{q}'}$ so $|R_{\arc{q},A}\cap R_{\arc{q}',A}|=0$.
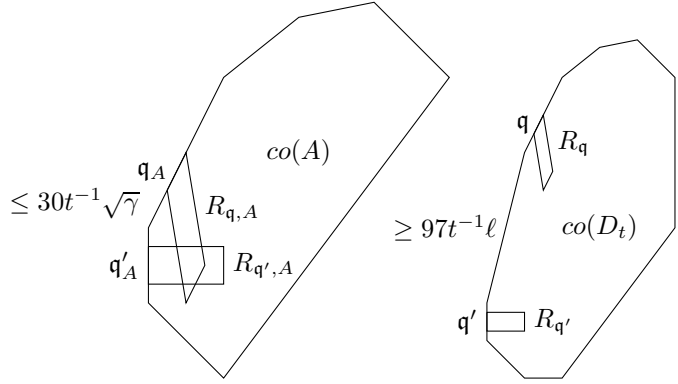
\begin{center}
\begin{tikzpicture}
\draw (0,0)--(-1,1)--(-1,2) node[anchor=south east] {$\le 30 t^{-1}\sqrt{\gamma}$}--(0,4) --(1,4.8)--(2,5)--(3,4)--cycle;

\draw (4,0)--(3.5,0.5)--(3.5,1)-- node[midway,left] {$\ge 97t^{-1}\ell$} (4,3) --(4.5,4)--(5,4.4)--(5.5,4.5)--(6,4)--(6,2)--(4.5,0)--cycle;

\draw (-1,1.75)--(0,1.75)-- node[midway, right] {$R_{\arc{q}',A}$} (0,1.25)--(-1,1.25)--(-1,1.75) node[midway,left] {$\arc{q}_A'$};
\draw (-0.75,2.5)--(-0.5,1)--(-0.25,1.5)-- node[midway, right] {$R_{\arc{q},A}$} (-0.5,3.0) -- node[midway,left] {$\arc{q}_A$}(-0.75,2.5);

\draw (3.5,0.875)--(4,0.875)--node[midway, right] {$R_{\arc{q}'}$} (4,0.625)--(3.5,0.625)--node[midway,left] {$\arc{q}'$}(3.5,0.875);

\draw (4.125,3.25)-- (4.25,2.5)--(4.375,2.75)-- node[midway, right] {$R_{\arc{q}}$} (4.25,3.5) -- node[midway,left] {$\arc{q}$}(4.125,3.25);

\draw (1,3) node {$co(A)$};
\draw (5,2) node{$co(D_t)$};
\end{tikzpicture}
\end{center}

Assume now that $\arc{r}\ne \arc{r}'$. In this case, the distance between $\arc{q}$ and $\arc{q}'$ is at least $97t^{-1}\ell$. Indeed, otherwise there exists a point $p\in \arc{q}$ and $p' \in \arc{q'}$ such that $|pp'| \le 97t^{-1}\ell$. Let $x$ be a $(\theta,\ell)$-bad point such that $|xp|\le 3\ell$. Then $B(x,100t^{-1}\ell)$ contains $p$, and by the triangle inequality it also contains $p'$. This implies $p,p'$ are contained in the same arc of $\hbad$, so $\arc{r}=\arc{r'}$, a contradiction.

Assuming for the sake of contradiction that $|R_{\arc{q},A}\cap R_{\arc{q}',A}|>0$, then there exists a point $z\in R_{\arc{q},A}\cap R_{\arc{q}',A}$. Then because $z$ is within distance $t^{-1}||v_{\arc{q}}||=15t^{-1   }\sqrt{\gamma}$ of $\arc{q}_A$ and within distance $t^{-1}||v_{\arc{q'}}||=15t^{-  1 }\sqrt{\gamma}$ of $\arc{q}_A'$, we have by the triangle inequality that the distance between $\arc{q}_A$ and $\arc{q}_A'$ is at most $30t^{-1}\sqrt{\gamma}\le 30t^{-1}\ell$.

By the above, there either exists $p\in \arc{q}$ and $z_A \in \arc{q}_A$ such that $|pz_A| \ge 33t^{-1}\ell$, or there exists $p' \in \arc{q'}$ and $z_A' \in \arc{q}_A'$ such that $|p'z_A'| \ge 33t^{-1}\ell$. Suppose without loss of generality the first case holds. Then $p=tx_A+(1-t)y_B$ for some point $x_A \in \arc{q}$ and $y_B=p_{\arc{q},B}$, and $|x_Az_A| \le \xi t^{-1} \sqrt{\gamma}$ since this is an upper bound for the length of $\arc{q}_A$. Therefore, $$|x_Ay_B| \ge |x_Ap| \ge |pz|-|x_Az| \ge 20t^{-1}\ell,$$ so by \Cref{farisgoodlem}, $p\in\arc{q}\in \jgood$, a contradiction.

\end{proof}

\begin{proof}[Proof of \Cref{overlapping}]
The proof we give works verbatim for $B$ and $\mathcal{B}$, so we focus on the case with $A$ and $\mathcal{A}$. Assume $d_{\tau}$ is sufficiently small so that \Cref{tangentangle} is true, and such that $\frac{99}{100}K\subset \co(A),\co(B),\co(D_t)\subset K$ by \Cref{etascaleinside}. Fix one of the $M$ directions $v$. Consider all arcs $\arc{q}\in \jj\cap \mathcal{A}$ with the direction vector $\widehat{v_\arc{q}}=v$. Let $\arc{r}_A$ be the union of all the corresponding arcs $\arc{q}_A$. Note that $\arc{r}_A$ forms a connected arc of $\partial co(A)$. Let $x$ and $x'$ be the endpoints of this arc.

For any point $z\in \arc{r}_A$, we claim that $ |xz| \le \frac{9}{\sin(14^{\circ})}\operatorname{dist}(z,ox)$. Indeed, by \Cref{distlinepoint}, since $|xz| \le 9|oz|$ (this follows as the diameter of $\co(A)\subset T'$ is at most $\frac{2}{\sqrt{3}}$ by \Cref{lengthsandareas}, and $|oz| \ge \frac{99}{100}\frac{1}{\sqrt{12}}$) it suffices to show that $\angle ozx \in (28^{\circ},180^{\circ}-28^{\circ})$. By \Cref{tangentangle}, we know that the supporting lines $l_x,l_z$ to $\co(A)$ at $x,z$ make an angle of at most $180^{\circ}-29^{\circ}$ with $ox,oz$ respectively. Therefore, we have that $\angle ozx, oxz \le 180^{\circ}-29^{\circ}$.  By \Cref{halfalpha}, $ox,oz$ each make an angle of at most $\frac{1}{2}\alpha$ with $v$. Therefore, $\angle xoz \le \alpha$. Because the sum of the angles in $xoz$ is  $180^{\circ}$, this implies that $\angle ozx \in (29^{\circ}-\alpha,180^{\circ}-29^{\circ})\subset (28^{\circ},180^{\circ}-28^{\circ})$.

For every $y$ outside of $\arc{r}_A$, we have either $y$ is on the opposite side of $ox$ or $y$ is on the opposite side of $oy$ to $\arc{r}_A$.
This implies that $\min(zx,zx')\le \frac{9}{\sin(14^{\circ})}|yz|$ as $y$ lies either on the other side of $ox$ or of $ox'$ to $z$.

We claim that if $R_{\arc{q},A}$ with $\arc{q}_A\subset \arc{r}_A$ intersects in positive area with some $R_{\arc{q}',A}$, then $\arc{q}_A,\arc{q}'_A\subset (B(x,1200t^{-1}\sqrt{\gamma})\cup B(x',1200 t^{-1}\sqrt{\gamma}))$.
Indeed, first note that if $\arc{q}_A'\subset \arc{r}_A$, then $\widehat{v_{\arc{q}}}=\widehat{v_{\arc{q}'}}$, forbidding a positive area intersection. Hence $\arc{q}_A$ lies outside of $\arc{r}_A$. Note that if $|R_{\arc{q},A}\cap R_{\arc{q}',A}|>0$, then the distance between $\arc{q}_A$ and $\arc{q}_A'$ is at most $30t^{-1}\sqrt{\gamma}$ by the triangle inequality (as the heights of these parallelograms are each at most $15t^{-1}\sqrt{\gamma}$). From this, we conclude that  \begin{align*}
\min(\operatorname{dist}(\arc{q}_A,x),\operatorname{dist}(\arc{q}_A,x')) &\le \frac{9}{\sin(14^{\circ})}30t^{-1}\sqrt{\gamma}\le 1199t^{-1}\sqrt{\gamma}.
\end{align*}
Because $$|\arc{q}_A| \le \xi t^{-1}\sqrt{\gamma} \le t^{-1}\sqrt{\gamma},$$
the conclusion follows.

We have the length of $\partial \co(A) \cap (B(x,1200t^{-1}\sqrt{\gamma})\cup B(x',1200t^{-1}\sqrt{\gamma}))$ is at most $4800\pi t^{-1}\sqrt{\gamma}$, the sum of the perimeters of the two balls. Hence for each direction $v$ we have that $$\sum_{\arc{q}\in\jgood\cap \mathcal{A},\widehat{v_{\arc{q}}}=v\text{ and } \exists  \arc{q}'\in \mathcal{A}\setminus \{\arc{q}\}\text{ with } |R_{\arc{q},A}\cap R_{\arc{q}',A}|>0}\big|R_{\arc{q},A}\big|\leq 4800\pi t^{-1}\sqrt{\gamma}\cdot \xi \sqrt{\gamma}=16000 t^{-1}\xi \gamma.$$
\end{proof}

\section{Proof of \Cref{mainthm} and \Cref{mainthm'}}
\label{puttingtogethersection}
With all the machinery in place, we are now ready to tackle  \Cref{mainthm'}. We note that \Cref{mainthm} and \Cref{mainthm'} are formally equivalent by replacing $A$ with $\frac{1}{t}A$ and $B$ with $\frac{1}{1-t}B$.

\begin{proof}[Proof of \Cref{mainthm'}]
Fix $\epsilon >0 $ and choose $\xi$ such that $\epsilon\geq (t^2+(1-t)^2)(25t^{-1}M\xi^2+16000t^{-1} M\xi)$. Choose $\theta$ depending on $\xi$ given by \Cref{guarenteelem}. Choose $\ell$ depending on $\theta$ given by \Cref{partitionlem}. Recall that $M, \alpha$ are universal constants chosen above. Finally, take $d_\tau$ sufficiently small so that \Cref{partitionlem}, \Cref{containedloss}, \Cref{notcontained}, \Cref{disjointbad} and \Cref{overlapping} hold.  Recall by \Cref{containedloss} that 
$$\bigg|\co\left(D_t\right)\setminus D_t\bigg|\leq t^{2}\sum_{\arc{q} \in \mathcal{A}}|R_{\arc{q},A}\setminus A|+(1-t)^2\sum_{\arc{q}\in \mathcal{B}}|R_{\arc{q},B}\setminus B|.$$
We split the first summand on the right into three parts; one for those $\arc{q}$ such that $R_{\arc{q},A}\not\subset co(A)$ (collect them in a set $X_A$), one for those $\arc{q}\in \jgood$ such that $R_{\arc{q},A}$ intersects non trivially with $R_{\arc{q},A}$ for some $\arc{q}'\neq \arc{q}$ (collect them in a set $Y_A$), and all the other $\arc{q}$ (collect them in a set $Z_A$). Note that the $R_{\arc{q},A}$ in the last sum are disjoint by \Cref{disjointbad} and contained in $co(A)$, so $\sum_{\arc{q}\in Z_A}|R_{\arc{q},A}\setminus A|\leq |co(A)\setminus A|$.
Combining \Cref{notcontained} and \Cref{overlapping} we find:
\begin{align*}
\sum_{\arc{q}\in \mathcal{A}}|R_{\arc{q},A}\setminus A|&\leq \sum_{\arc{q}\in X_A}|R_{\arc{q},A}|+\sum_{\arc{q}\in Y_A}|R_{\arc{q},A}|+\sum_{\arc{q}\in Z_A}|R_{\arc{q},A}\setminus A| \\
&\leq 25t^{-1}M\xi^2\gamma+16000t^{-1} M\xi \gamma+|co(A)\setminus A|.
\end{align*}
We similarly obtain
$$\sum_{\arc{q}\in \mathcal{B}}|R_{\arc{q},B}\setminus B|\le 25t^{-1}M\xi^2\gamma+16000t^{-1} M\xi \gamma+|\co(B)\setminus B|.$$

Hence, (recalling $\gamma=t^2|\co(A)\setminus A|+(1-t)^2|\co(B)\setminus B|$), we have
\begin{align*}\bigg|\co\left(D_t\right)\setminus D_t\bigg| &\leq (t^2+(1-t)^2)(25t^{-1}M\xi^2+16000t^{-1}M\xi)\gamma+t^2|\co(A)\setminus A|+(1-t)^2|\co(B)\setminus B|\\ &\leq \left(1+\epsilon\right)\big(t^2|\co(A)\setminus A|+(1-t)^2|\co(B)\setminus B|\big).
\end{align*}
\end{proof}

\section{Proof that \Cref{mainthm} implies \Cref{sharpBM}}
\label{implicationsection}
Finally, what remains is to deduce \Cref{sharpBM}. Note that we now return to $A$ and $B$ with unequal areas.

\begin{proof}[Proof that \Cref{mainthm} implies \Cref{sharpBM}]

 By \cite{Figalli09,Figalli10amass} and \Cref{OmegaEquiv}, there is a constant $\widetilde{C}$ such that
$$\frac{|K_A\setminus \co(A)|}{|\co(A)|}+\frac{|K_B\setminus \co(B)|}{|\co(B)|} \le \widetilde{C}\tau^{-\frac{1}{2}}_{conv}\sqrt{\delta_{conv}}$$ where
$\delta_{conv}=\frac{|\co(A+B)|^{\frac{1}{2}}}{|\co(A)|^{\frac{1}{2}}+|\co(B)|^{\frac{1}{2}}}-1,$ and $t_{conv}=\frac{|\co(A)|^{\frac{1}{2}}}{|\co(A)|^{\frac{1}{2}}+|\co(B)|^{\frac{1}{2}}} \in [\tau_{conv},1-\tau_{conv}]$. Also, by \Cref{FigJerBm} by taking $d_\tau$ sufficiently small, we may assume that $\frac{|\co(A)|}{|A|}$,  $\frac{|\co(B)|}{|B|}$, and $\frac{|\co(A+B)|}{|A+B|}$ are as close to $1$ as we like, so in particular we may assume that $\tau_{conv}^{-1}\le 2\tau^{-1}$. Thus it suffices to prove that $\delta_{conv}\le \delta$ and $\frac{|\co(A)\setminus A|}{|\co(A)|}+\frac{|\co(B)\setminus B|}{|\co(B)|} \le 5\tau^{-1}\delta$. We have
\begin{align*}& \delta-\delta_{conv}\\
&\ge \frac{|A|^{\frac{1}{2}}+|B|^{\frac{1}{2}}}{|\co(A)|^{\frac{1}{2}}+|\co(B)|^{\frac{1}{2}}}\delta-\delta_{conv}\\
&=\frac{1}{{|\co(A)|^{\frac{1}{2}}+|\co(B)|^{\frac{1}{2}}}}\left(|\co(A)|^{\frac{1}{2}}-|A|^{\frac{1}{2}}+|\co(B)|^{\frac{1}{2}}-|B|^{\frac{1}{2}}-(|\co(A+B)|^{\frac{1}{2}}-|A+B|^{\frac{1}{2}})\right)\\
&=\frac{1}{{|\co(A)|^{\frac{1}{2}}+|\co(B)|^{\frac{1}{2}}}}\left(\frac{|co(A)\setminus A|}{|\co(A)|^{\frac{1}{2}}+|A|^{\frac{1}{2}}}+\frac{|\co(B)\setminus B|}{|\co(B)|^{\frac{1}{2}}+|B|^{\frac{1}{2}}}-\frac{|\co(A+B)\setminus(A+B)|}{|\co(A+B)|^{\frac{1}{2}}+|A+B|^{\frac{1}{2}}}\right)\\
&\ge\frac{1}{|\co(A)|^{\frac{1}{2}}+|\co(B)|^{\frac{1}{2}}}\left(\frac{|\co(A)\setminus A|}{|\co(A)|^{\frac{1}{2}}+|A|^{\frac{1}{2}}}+\frac{|\co(B)\setminus B|}{|\co(B)|^{\frac{1}{2}}+|B|^{\frac{1}{2}}}-\frac{(1+\epsilon)(|\co(A)\setminus A|+|\co(B)\setminus B|)}{|\co(A+B)|^{\frac{1}{2}}+|A+B|^{\frac{1}{2}}}\right).
\end{align*}
Suppose $t\le \frac{1}{2}$ and take $\epsilon=\frac{\tau}{2}$. We can write this last line as $m_A\frac{|\co(A)\setminus A|}{|\co(A)|}+m_B\frac{|\co(B)\setminus B|}{|\co(B)|}$ with
\begin{align*}m_A=&t\frac{|\co(A)|}{|A|}\cdot\frac{|A|^{\frac{1}{2}}+|B|^{\frac{1}{2}}}{|\co(A)|^{\frac{1}{2}}+|\co(B)|^{\frac{1}{2}}}\left(\frac{1}{\frac{|\co(A)|^{\frac{1}{2}}}{|A|^{\frac{1}{2}}}+1}-\frac{1}{\frac{|\co(A+B)|^{\frac{1}{2}}}{|A+B|^{\frac{1}{2}}}+1}\cdot\frac{(1+\epsilon)t}{(1+\delta)}\right)\\
&\ge t\frac{|\co(A)|}{|A|}\cdot\frac{|A|^{\frac{1}{2}}+|B|^{\frac{1}{2}}}{|\co(A)|^{\frac{1}{2}}+|\co(B)|^{\frac{1}{2}}}\left(\frac{1}{\frac{|\co(A)|^{\frac{1}{2}}}{|A|^{\frac{1}{2}}}+1}-\frac{1}{\frac{|\co(A+B)|^{\frac{1}{2}}}{|A+B|^{\frac{1}{2}}}+1}\cdot\frac{3}{4}\right) \end{align*}
and
\begin{align*}m_B&=(1-t)\frac{|\co(B)|}{|B|}\cdot\frac{|A|^{\frac{1}{2}}+|B|^{\frac{1}{2}}}{|\co(A)|^{\frac{1}{2}}+|\co(B)|^{\frac{1}{2}}}\left(\frac{1}{\frac{|\co(B)|^{\frac{1}{2}}}{|B|^{\frac{1}{2}}}+1}-\frac{1}{\frac{|\co(A+B)|^{\frac{1}{2}}}{|A+B|^{\frac{1}{2}}}+1}\cdot\frac{(1+\epsilon)(1-t)}{(1+\delta)}\right)\\
&\ge (1-t)\frac{|\co(B)|}{|B|}\cdot\frac{|A|^{\frac{1}{2}}+|B|^{\frac{1}{2}}}{|\co(A)|^{\frac{1}{2}}+|\co(B)|^{\frac{1}{2}}}\left(\frac{1}{\frac{|\co(B)|^{\frac{1}{2}}}{|B|^{\frac{1}{2}}}+1}-\frac{1}{\frac{|\co(A+B)|^{\frac{1}{2}}}{|A+B|^{\frac{1}{2}}}+1}\cdot (1-\frac{\tau}{2})\right).
\end{align*}
Both of these are at least $\frac{1}{5}\tau$ assuming $d_\tau$ is sufficiently small. Thus we get $\delta-\delta_{conv} \ge \frac{1}{5}\tau(\frac{|\co(A)\setminus A|}{|\co(A)|}+\frac{|\co(B)\setminus B|}{|\co(B)|})$, which shows $\delta_{conv} \le \delta$ and $\frac{|\co(A)\setminus A|}{|\co(A)|}+\frac{|\co(B)\setminus B|}{|\co(B)|} \le 5\tau^{-1}\delta$.
\end{proof}

\appendix

\section{Equivalence of measures $\omega$ and $\alpha$}
\label{OmegaEquiv}
In this appendix, we show that in two dimensions the measures $\omega$ and $\alpha$ are commensurate for convex sets when $d_\tau$ is sufficiently small. Recall from the introduction that we always have $\alpha \le 2\omega$.

\begin{prop} For all $\tau\in (0,\frac12]$, there exists a $d_\tau>0$ such that the following holds. If $E,F\subset \mathbb{R}^2$ are convex with $t(E,F)\in [\tau,1-\tau]$ and $\delta(E,F)\leq d_\tau$, then
$$\omega(E,F)\leq 21 \alpha(E,F).$$
\end{prop}

\begin{proof}
Let $d_\tau$ be sufficiently small so that by \cite{Figalli09}, $\alpha(E,F)\le \frac{1}{10}$. We never use any other property of $\delta(E,F)$ or $t(E,F)$. The quantitites $\omega,\alpha$ are invariant under affine transformations of $E$ and $F$ separately, so by applying these transforms we can take $E,F$ to have equal volumes, translated so that $\alpha(E,F)=\frac{|E\Delta F|}{|E|}$. After a further affine transformation, we may assume that the maximal triangle $T\subset E\cap F$ is a unit equilateral triangle. Note that because $T$ is maximal, we have $T\subset E\cap F \subset -2T$. Take $K=\co(E \cup F)$. Note that $|E\Delta F| \le \frac{1}{18}|E\cap F| \le \frac{1}{18}|-2T| \le \frac12$.



First, we claim that $E,F\subset 10C$. Indeed, if any point $x\in E$ lies in in $\partial 10T$ then $|E\Delta F| \ge |co(x \cup T)\setminus (-2T)| \ge 1$ , a contradiction.

To show $\omega(E,F)\le 11\alpha(E,F)$, it suffices to prove
$$|(K\setminus (A \cup B))|\le 10|(A \Delta B)|.$$
Indeed, if this is true, then
$$|E|\cdot\omega(E,F)\leq |K\setminus E|+|K\setminus F|=2|K\setminus (E\cup F)|+|E \Delta F| \le 21|E \Delta F|= |E| \cdot 21\alpha (E,F).$$

We consider the triangle $opq$ with $p,q$ consecutive vertices of $K$. These triangles partition the area of $K$, so it suffices to show for each such triangle that $$|(K\setminus (E \cup F))\cap opq|\le 10|(E \Delta F)\cap opq|.$$

To obtain this, we note that if $p,q \in E$ or $p,q \in F$ then the left hand side is zero and the inequality holds. Suppose now that $p \in E$ and $q \in F$ (the other case is identical). Then there must be a point $i\in \partial co(A)\cap \partial co(F)$ which lies in the triangle $opq$. Let $q'$ be the intersection of the ray $pi$ with segment $oq$, and let $p'$ be the intersection of the ray $qi$ with $op$. Because $o,p \in E$ we also have $p' \in E$, and similarly $q' \in F$. We note that $E,F\subset 10C$ implies $|op'| \ge \frac{1}{10}|oq|$ and $|oq'| \ge \frac{1}{10}|oq|$.
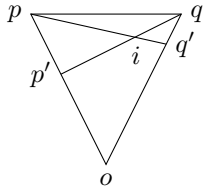
\begin{center}
\begin{tikzpicture}
\draw (-1,2) node[anchor=east] {$p$}--(0,0) node[anchor=north] {$o$}--(1,2) node[anchor=west] {$q$}--cycle;

\draw (0.4,1.45) node{$i$};

\draw (-1,2) --(0.8,1.6) node[anchor=west] {$q'$};
\draw (1,2)--(-0.6,1.2) node[anchor=east] {$p'$};
\end{tikzpicture}
\end{center}
If any point $x$ in the strict interior  $(qiq')^{\circ}$ lies in $E$, then $i$ lies in the strict interior of $xpo\subset E$, contradicting that $i$ lies on $\partial E$. Also, $qiq'\subset oqi \subset F$. Thus $(qiq')^{\circ}\subset E\Delta F$. Similarly  $(pip')^{\circ}\subset E \Delta F$. Finally, we note that $(K\setminus (E\cup F))\cap opq\subset piq$, so it suffices to show that
$$|piq|\le 10(|pip'|+|qiq'|).$$

To show this, suppose without loss of generality that $|oiq| \le |oip|$. Then 
\begin{align*}
    \frac{|piq|}{|oiq|}=\frac{|pip'|}{|oip'|}
\end{align*}
so
$$|piq|=|pip'| \frac{|oiq|}{|oip'|}\le |pip'|\frac{|oip|}{|oip'|}=|pip'|\frac{|op|}{|op'|}\le 10|pip'|.$$

\end{proof}

\bibliographystyle{abbrv}
\bibliography{references}

\end{document}